\documentclass{article}
\usepackage[utf8]{inputenc}

\usepackage{amsmath,amssymb,amsthm}
\usepackage{bm}
\usepackage{algorithm,algorithmic}
\usepackage{enumerate}
\usepackage[margin=25mm]{geometry}
\usepackage[numbers]{natbib}
\usepackage{mathtools}
\mathtoolsset{showonlyrefs}
\usepackage{hyperref}
\usepackage{color}

\newcommand{\argmin}{\mathop{\rm argmin}\limits}
\newcommand{\dom}{\mathop{\rm dom}}
\newcommand{\interior}{\mathop{\rm int}}

\newcommand{\closure}{\mathop{\rm cl}}

\newcommand{\dist}{\mathrm{dist}}

\newcommand{\innerprod}[2]{\langle #1,#2 \rangle}

\theoremstyle{definition}
\newtheorem{theorem}{Theorem}[section]

\newtheorem{lemma}{Lemma}[section]
\newtheorem{definition}{Definition}[section]
\newtheorem{remark}{Remark}[section]
\newtheorem{example}{Example}[section]
\newtheorem{proposition}{Proposition}[section]
\newtheorem{assumption}{Assumption}[section]

\title{Proximal gradient-type method with generalized distance and convergence analysis without global descent lemma}
\author{Shotaro Yagishita\thanks{Risk Analysis Research Center, The Institute of Statistical Mathematics, Japan. E-mail: syagi@ism.ac.jp} \thanks{Center for Social Data Structuring, Joint Support-Center for Data Science Research, Japan}
\and
Masaru Ito\thanks{Department of Mathematics, College of Science and Technology, Nihon University, Japan. E-mail: ito.masaru@nihon-u.ac.jp}}
\date{\today}

\begin{document}

\maketitle

\begin{abstract}
We consider solving nonconvex composite optimization problems in which the sum of a smooth function and a nonsmooth function is minimized.
Many of convergence analyses of proximal gradient-type methods rely on global descent property between the smooth term and its proximal term.
On the other hand, the ability to efficiently solve the subproblem depends on the compatibility between the nonsmooth term and the proximal term.
Selecting an appropriate proximal term by considering both factors simultaneously is generally difficult.
We overcome this issue by providing convergence analyses for proximal gradient-type methods with general proximal terms, without requiring global descent property of the smooth term.
As a byproduct, new convergence results of the interior gradient methods for conic optimization are also provided.
\end{abstract}

\section{Introduction}\label{sec:intro}
We consider solving nonconvex composite optimization problems in which the sum of a smooth function $f$ and a nonsmooth function $g$ is minimized.
As many problems in machine learning, signal processing, and statistical inference can be formulated as them, the composite problems have garnered interest.

The proximal gradient method (PGM), which was originally introduced by \citet{fukushima1981generalized}, is a standard algorithm for such composite problems.
The PGM updates the sequence $\{x^k\}$ by
\begin{equation}
    x^{k+1}\in\argmin_x\left\{\innerprod{\nabla f(x^k)}{x}+\frac{L_k}{2}\|x-x^k\|^2+g(x)\right\}
\end{equation}
with an appropriate choice of $L_k>0$.
The squared norm plays a role of a proximal term, ensuring that the updated point remains close to the previous one.
Many convergence analyses of the PGM rely on the global descent lemma, which is implied by the Lipschitz continuity of $\nabla f$ (see, e.g., \citep{beck2017first}).
Namely, they assume the existence of $L>0$ such that
\begin{equation}
    f(x)\le f(y)+\innerprod{\nabla f(y)}{x-y}+\frac{L}{2}\|x-y\|^2
\end{equation}
holds for any $x$ and $y$.
As the establishment of the global descent lemma is a restrictive assumption, \citet{bolte2018first} have provided convergence analysis of the Bregman proximal gradient method (BPGM) for nonconvex composite problems without the global Lipschitz assumption.
The BPGM iterates
\begin{equation}
    x^{k+1}\in\argmin_x\left\{\innerprod{\nabla f(x^k)}{x}+L_kD_h(x,x^k)+g(x)\right\},
\end{equation}
where $D_h(x,y)\coloneqq h(x)-h(y)-\innerprod{\nabla h(y)}{x-y}$ is the Bregman divergence generated by a strictly convex function $h$.
That is, the BPGM exploits the Bregman divergence as a proximal term instead of the squared norm.
Their analysis requires the relative smoothness of $f$, which is an extension of the global descent lemma \citep{van2017forward,bauschke2017descent,lu2018relatively}.
The function $f$ is said to be $L$-smooth relative to $h$ if it holds that
\begin{equation}
    f(x)\le f(y)+\innerprod{\nabla f(y)}{x-y}+LD_h(x,y)
\end{equation}
for any $x$ and $y$.
Since the BPGM and the relative smoothness reduce to the PGM and the global descent lemma when $h(x)=\frac{1}{2}\|x\|^2$ is used as a kernel, the BPGM can be viewed as a generalization of the PGM.

On the other hand, for the BPGM to be applied efficiently, it is necessary that the solution of its subproblem can be computed easily.
Since $f$ is linearized in the subproblem, the compatibility between $g$ and the proximal term becomes important.
Thus, in order to ensure that $f$ satisfies the relative smoothness condition and that the subproblem can be solved easily, one must choose a proximal term that is well-suited to both $f$ and $g$.
Unfortunately, it is generally challenging to choose such a proximal term.

Recently, \citet{kanzow2022convergence} have established a convergence analysis of the PGM without relying on a global descent lemma, and several subsequent studies have followed \citep{de2022proximal,de2023proximal,jia2023convergence}.
Building on these developments, this paper makes the following contributions.

\begin{itemize}
    \item We develop a unified convergence framework for proximal gradient-type methods with proximal terms more general than the Bregman divergence, without assuming a global descent lemma.
    Under this framework, we establish subsequential convergence results (Theorems \ref{thm:global-convergence-F-stat} and \ref{thm:global-convergence-M-stat}) and a global convergence result under the Kurdyka--\L ojasiewicz property (Theorem \ref{thm:KL}).
    As a byproduct, our analysis also provides a convergence guarantee for the BPGM without requiring relative smoothness.
    \item Our analysis enables the proximal term to be designed solely based on its compatibility with $g$, without requiring compatibility with $f$, thereby substantially enlarging the class of admissible proximal structures.
    In Section \ref{sec:application}, we exploit this flexibility to develop proximal gradient-type methods beyond the BPGM framework and demonstrate that our theory directly yields convergence guarantees for such methods.
    These results are not merely illustrative; rather, they highlight classes of algorithms that could not be covered by existing analyses based on Bregman-type proximal terms or relative smoothness.
    As a particularly important consequence, we revisit interior gradient methods for conic optimization \citep{auslender2004interior,auslender2006interior} under the proposed convergence framework and establish convergence results in the nonconvex setting.
    To the best of our knowledge, this provides the first convergence analysis for interior gradient methods beyond convexity.
\end{itemize}

The rest of this paper is organized as follows.
The remainder of this section is devoted to related works, notation, and preliminary results.
In the next section, we introduce our proposed algorithm, the generalized variable distance proximal gradient method and show some convergence results, which includes an analysis in the presence of the Kurdyka--\L ojasiewicz property.
Section \ref{sec:application} is devoted to applications of our results.
Finally, Section \ref{sec:conclusion} concludes the paper with some remarks.

\subsection{Related works}
The convergence analysis of the proximal quasi-Newton-type method \citep{chen1997convergence,becker2012quasi,lee2014proximal}, which utilizes variable metrics at each iteration and is an extension of the PGM, without assuming the global descent lemma was provided by \citet{tseng2009coordinate} only in the case where $g$ is convex.
\citet{hua2016block} showed the convergence result of the BPGM using different Bregman divergences at each iteration.
\citet{bonettini2016variable} established the convergence analysis of the proximal gradient-type methods using generalized distances as the proximal term.
Although the results of \citet{hua2016block} and \citet{bonettini2016variable} also do not require any global descent assumption, do require the convexity of $g$.
The first analysis of the PGM that requires neither the convexity of $g$ nor the global descent lemma has been conducted by \citet{kanzow2022convergence}.
\citet{de2022proximal} and \citet{de2023proximal} have established such an analysis for variants of the PGM.
\citet{jia2023convergence} have provided convergence of whole sequence and rate of convergence for the PGM under the Kurdyka--\L ojasiewicz property and the local Lipschitz assumption.
Note that all of the above utilize the backtracking strategy to determine the stepsize.
Their proof techniques revisit classical arguments due to \citet{gafni1982convergence} and \citet{bertsekas1999nonlinear}, which we also follow.

\subsection{Notation and Preliminaries}
For a positive integer $n$, the set $[n]$ is defined by $[n]\coloneqq\{1,\ldots,n\}$.
Let $\mathbb{E}$ be a finite-dimensional inner product space endowed with an inner product $\innerprod{\cdot}{\cdot}$.
The induced norm is denoted by $\|\cdot\|$.
For a matrix $X\in\mathbb{R}^{m\times n}$, $\|X\|_1$ denotes the $\ell_1$ norm that is defined by $\|X\|_1\coloneqq\sum_{j=1}^m\sum_{j'=1}^n|X_{jj'}|$.
We denote the set of nonnegative real numbers and the set of positive numbers by $\mathbb{R}_+$ and $\mathbb{R}_{++}$, respectively.
For a subset $\mathcal{A}\subset\mathbb{E}$, its interior and its closure are denoted by $\interior\mathcal{A}$ and $\closure\mathcal{A}$, respectively.
We denote the closed ball with center $x\in\mathbb{E}$ and radius $\rho$ by $\mathcal{B}_\rho(x)$.
For $\mathcal{A}\subset\mathbb{E}$, $\delta_\mathcal{A}:\mathbb{E}\to\{0,\infty\}$ denotes the indicator function of $\mathcal{A}$.

Let $\phi:\mathbb{E}\to(-\infty,\infty]$ be a function.
The domain of $\phi$ is denoted by $\dom \phi\coloneqq\{x\in\mathbb{E}\mid\phi(x)<\infty\}$.
A function $\phi$ is said to be coercive if $\lim_{\|x\|\to\infty}\phi(x)=\infty$.
If $\lim_{\|x\|\to\infty}\phi(x)/\|x\|=\infty$, we say that $\phi$ is supercoercive.
For $x\in\dom\phi$,
\begin{equation}
    \widehat{\partial}\phi(x) \coloneqq \left\{g\in\mathbb{E}~\middle|~\liminf_{y\to x}\frac{\phi(y)-\phi(x)-\innerprod{g}{y-x}}{\|y-x\|}\ge0\right\}
\end{equation}
is called the Fr\'echet subdifferential of $\phi$ at $x$ and
\begin{equation}
    \partial\phi(x) \coloneqq \left\{g\in\mathbb{E}~\middle|~\exists\{x^k\},\{g^k\}~ \mbox{s.t.}~ x^k\to x,~ \phi(x^k)\to\phi(x),~ g^k\to g,~ g^k\in\widehat{\partial}\phi(x^k)\right\}
\end{equation}
is known as the Mordukhovich subdifferential of $\phi$ at $x$.
We call a point $x^*\in\dom\phi$ satisfying $0\in\widehat{\partial}\phi(x^*)$ (resp. $0\in\partial\phi(x^*)$) an F-stationary point (resp. M-stationary point) of $\min_{x\in\mathbb{E}}\phi(x)$.
If $\phi$ is of the form $\phi=\phi_1+\phi_2$ where $\phi_1$ is continuously differentiable, it holds that $\widehat{\partial}\phi(x)=\nabla\phi_1(x)+\widehat{\partial}\phi_2(x)$ and $\partial\phi(x)=\nabla\phi_1(x)+\partial\phi_2(x)$ for $x\in\dom\phi_2$ \citep[Exercise 8.8]{rockafellar2009variational}.
It is clear that $\widehat{\partial}\phi(x)\subset\partial\phi(x)$ holds, which implies that the F-stationarity is sharper than the M-stationarity.
However, since the M-stationarity is commonly used and popular, we consider the convergence for both.

The subderivative of $\phi$ at $x\in\dom\phi$ in direction $d$ is defined by
\begin{equation}\label{eq:subderivative}
    \phi'(x;d)\coloneqq\liminf_{\substack{\eta\searrow 0\\d'\to d}}\frac{\phi(x+\eta d')-\phi(x)}{\eta}=\liminf_{\substack{\eta\searrow 0\\d'\to d\\x+\eta d'\in\dom\phi}}\frac{\phi(x+\eta d')-\phi(x)}{\eta}.
\end{equation}
It immediately follows from \citep[Exercise 8.4]{rockafellar2009variational} that $x^*\in\dom\phi$ is an F-stationary point if and only if $\phi'(x^*;d)\ge0$ holds for all $d\in\mathcal{T}(x^*;\dom\phi)$, where $\mathcal{T}(x^*;\dom\phi)$ is the (Bouligand) tangent cone of $\dom\phi$ at $x^*$, that is,
\begin{align}
    \mathcal{T}(x^*;\dom\phi)\coloneqq
        \left\{d\in\mathbb{E}~\middle|~\exists\{d^k\},\{\eta_k\}~ \mbox{s.t.}~  x^*+\eta_k d^k\in\dom\phi,~ d^k\to d,~ \eta_k\searrow0\right\}.
\end{align}

The Kurdyka--\L ojasiewicz (KL) property is defined as follows.

\begin{definition}[\citep{attouch2010proximal,attouch2013convergence,bolte2014proximal}]\label{def:KL}
For a lower semicontinuous function $\Phi:\mathbb{E}\to(-\infty,\infty]$, we say that $\Phi$ has the KL property at $x^*\in\dom\partial \Phi$ if there exists a positive constant $\varpi$, a neighborhood $\mathcal{U}$ of $x^*$, and a continuous concave function $\chi:[0,\varpi)\to[0,\infty)$ that is continuously differentiable on $(0,\varpi)$ and satisfies $\chi(0)=0$ as well as $\chi'(t)>0$ on $(0,\varpi)$, such that
\begin{equation}
    \chi'(\Phi(x)-\Phi(x^*))~\dist(0,\partial \Phi(x))\ge1
\end{equation}
holds for all $x\in\mathcal{U}$ satisfying $\Phi(x^*)<\Phi(x)<\Phi(x^*)+\varpi$.
We refer to $\chi$ as the \emph{desingularization function}.
If the desingularization function is of the form $\chi(t)=ct^\theta$ where $c>0$ and $0<\theta\le1$, then we say that $\Phi$ has the KL property at $x^*$ with an exponent of $\theta$\footnote{While $1-\theta$ is more commonly referred to as the exponent, in this work we adopt $\theta$, following the convention of \citet{jia2023convergence}.}.
\end{definition}

The KL property is often used in the analysis of first-order methods to examine the convergence of the entire sequence and the convergence rate, and it is also used for this purpose in this paper.
Moreover, wide classes of functions admitting the KL property are known including semialgebraic or subanalytic ones (see, e.g., \cite{shiota1997,bolte2007KL,li2017KL} and references therein).
For instance, any subanalytic function $\Phi$ with a closed domain on which $\Phi$ is continuous satisfies the KL property at any $x^* \in \dom \partial\Phi$ for some exponent \cite[Theorem 3.1 and Remark 3.2]{bolte2007KL}.

We first introduce a distance-like function used as a proximal term.
The \emph{prox-grad distance} with respect to an open convex set $\mathcal{C}\subset\mathbb{E}$ is defined as follows.

\begin{definition}\label{def:prox-distance}
A nonnegative-valued function $D:\mathbb{E}\times\mathcal{C}\to[0,\infty]$ is called a prox-grad distance if the following conditions hold for any $y\in\mathcal{C}$:
\begin{enumerate}[(i)]
    \item $D(x,y)=0$ if and only if $x=y$;
    \item The function $D(\cdot,y)$ is a lower semicontinuous function with $\mathcal{C}\subset\dom D(\cdot,y)$;
    \item The function $D(\cdot,y)+\delta_{\closure\mathcal{C}}$ is supercoercive.
\end{enumerate}
\end{definition}

It is obvious that $D(x,y)=\frac{1}{2}\|x-y\|^2$ is a prox-grad distance with respect to an arbitrary open convex set $\mathcal{C}\subset\mathbb{E}$.
More generally, let us consider the Bregman divergence.
Let $h:\mathbb{E}\to(-\infty,\infty]$ be a lower semicontinuous strictly convex function being continuously differentiable on $\mathcal{C}$.
The Bregman divergence $D_h:\mathbb{E}\times\mathcal{C}\to[0,\infty]$ generated by $h$ is defined by
\begin{equation}\label{eq:Bregman-div}
    D_h(x,y)\coloneqq h(x)-h(y)-\innerprod{\nabla h(y)}{x-y}.
\end{equation}
If $h+\delta_{\closure\mathcal{C}}$ is supercoercive, the Bregman divergence \eqref{eq:Bregman-div} is a prox-grad distance with respect to $\mathcal{C}$.
The strong convexity of $h$, which is assumed in \citet{bolte2018first} for global convergence properties of the BPGM, implies the supercoerciveness of $h$.

The following notion is a generalization of the prox-boundedness (see, e.g., \citep[Definition 1.23]{rockafellar2009variational}) for the prox-grad distance.

\begin{definition}
We say that a function $\phi:\mathbb{E}\to(-\infty,\infty]$ is \emph{prox-bounded with respect to $D$} if there exists $\gamma>0$ such that $\phi+\gamma D(\cdot,y)$ is bounded from below on $\mathbb{E}$ for any $y\in\mathcal{C}$.
The infimum of the set of all such $\gamma$ is the threshold $\gamma_{\phi,D}$ of the prox-boundedness with respect to $D$ for $\phi$.
\end{definition}

The prox-boundedness with respect to $D(x,y)=\frac{1}{2}\|x-y\|^2$ coincides with the usual prox-boundedness.
The following proposition provides that a generalized proximal mapping is nonempty and compact.

\begin{proposition}\label{prop:nonemptyness-compactness}
Let $\phi:\mathbb{E}\to(-\infty,\infty]$ be lower semicontinuous and prox-bounded with respect to a prox-grad distance $D$ with threshold $\gamma_{\phi,D}$.
Suppose that $\dom\phi$ is included in $\closure\mathcal{C}$ and $\mathcal{C}\cap\dom\phi$ is nonempty.
Then, for any $y\in\mathcal{C},~ a\in\mathbb{E}$ and $\gamma>\gamma_{\phi,D}$,
\begin{equation}\label{eq:generalized-prox-map}
    \argmin_{x\in\mathbb{E}}\left\{\innerprod{a}{x}+\gamma D(x,y)+\phi(x)\right\}
\end{equation}
is nonempty and compact.
\end{proposition}

\begin{proof}
From the assumptions and Definition \ref{def:prox-distance}, the objective in the minimization problem \eqref{eq:generalized-prox-map} is proper and lower semicontinuous.
Let $\gamma^*\coloneqq(\gamma+\gamma_{\phi,D})/2 \in (\gamma_{\phi,D}, \gamma)$ and $l^*\in\mathbb{R}$ be a lower bound of $\phi+\gamma^*D(\cdot,y)$, then we have
\begin{align}
    \innerprod{a}{x}+\gamma D(x,y)+\phi(x) &\ge-\|a\|\|x\|+(\gamma-\gamma^*)D(x,y)+\delta_{\closure\mathcal{C}}(x)+l^*\\
    &=\|x\|\left\{\frac{(\gamma-\gamma^*)D(x,y)+\delta_{\closure\mathcal{C}}(x)}{\|x\|}-\|a\|\right\}+l^*,
\end{align}
and hence the objective is coercive by the supercoerciveness of $D(\cdot,y)+\delta_{\closure\mathcal{C}}$.
Thus, the set \eqref{eq:generalized-prox-map} is nonempty and compact (see, e.g., \citep[Theorem 1.9]{rockafellar2009variational}).
\end{proof}

\section{Generalized variable distance proximal gradient method}\label{sec:convergence}
We consider a proximal gradient-type algorithm for the following composite optimization problem
\begin{equation}\label{problem:general}
    \underset{x\in\mathbb{E}}{\mbox{minimize}} \quad F(x)\coloneqq f(x)+g(x),
\end{equation}
where $f:\mathbb{E}\to(-\infty,\infty]$ is continuously differentiable on $\interior\dom f$ and $g:\mathbb{E}\to(-\infty,\infty]$ is proper and lower semicontinuous.
Suppose that $F$ is bounded from below and lower semicontinuous.

Our proposed algorithm is a proximal gradient-type method using prox-grad distance with respect to an open convex set $\mathcal{C}$.
The algorithm is called \emph{generalized variable distance proximal gradient method} (for short GVDPGM) and summarized in Algorithm \ref{alg:GVDPGM}.
Our method allows the selection of different prox-grad distances at each iteration.
The stepsize at each iteration is determined by backtracking technique to satisfy average-type nonmonotone Armijo condition \citep{zhang2004nonmonotone}.

\begin{algorithm}[H]
\caption{GVDPGM for \eqref{problem:general}}
    \label{alg:GVDPGM}
    \begin{algorithmic}
    \STATE {\bfseries Input:} $x^0\in\mathcal{C}\cap\dom g,~ F_0=F(x^0),~ \beta>1,~ 0<\sigma<1,~ 0<p\le1$, and $k=0$.
    \REPEAT
    \STATE Choose a prox-grad distance $D_k$.
    \STATE Find the smallest $i\in\{0,1,2,\ldots\}$ s.t.
    \begin{equation}\label{eq:acceptance-criterion}
        F(x^{k,i})\le F_k-\sigma \beta^iD_k(x^{k,i},x^k)
    \end{equation}
    where
    \begin{equation}\label{eq:general-prox-map}
        x^{k,i}\in\argmin_{x\in\mathbb{E}}\left\{f(x^k)+\innerprod{\nabla f(x^k)}{x-x^k}+\beta^iD_k(x,x^k)+g(x)\right\}.
    \end{equation}
    \STATE Denote $i_k=i$ and choose $p_{k+1}\in[p,1]$.
    \STATE Set $x^{k+1}=x^{k,i_k}$, $F_{k+1}=p_{k+1}F(x^{k+1})+(1-p_{k+1})F_k$, and $k\leftarrow k+1$.
    \UNTIL Termination criterion is satisfied.
    \end{algorithmic}
\end{algorithm}

The GVDPGM reduces to the proximal quasi-Newton-type method when $D_k(x,y)=\frac{1}{2}\innerprod{x-y}{H_k(x-y)}$ with a symmetric positive definite operator $H_k:\mathbb{E}\to\mathbb{E}$.
If the Bregman divergence is used as a prox-grad distance, the GVDPGM coincides with the BPGM.

Below, we first provide the subsequential convergence of the GVDPGM and compare it with existing methods.
Then, we present a convergence analysis of the GVDPGM in the presence of the Kurdyka--\L ojasiewicz property.
The following assumptions are made throughout the paper.

\begin{assumption}\label{assume:well-definedness}
~
\begin{enumerate}[(i)]
    \item The function $f:\mathbb{E}\to(-\infty,\infty]$ is continuously differentiable on $\interior\dom f$, $g:\mathbb{E}\to(-\infty,\infty]$ is proper and lower semicontinuous, and $F$ is bounded from below and lower semicontinuous;
    \item It holds that $\mathcal{C}\subset\interior\dom f$, $\dom F\subset\interior\dom f$, $\dom g\subset\closure\mathcal{C}$, and $\mathcal{C}\cap\dom g\neq\emptyset$;
    \item The function $\nabla f$ is locally Lipschitz continuous on $\interior\dom f$, equivalently, $\nabla f$ is Lipschitz continuous on any compact subset of $\interior\dom f$;
    \item For all $k\ge0$, $g$ is prox-bounded with respect to $D_k$ with threshold $\gamma_{g,D_k}<1$;
    \item For all $k\ge0$ and $y\in\mathcal{C}$, there exist a positive number $\alpha$ and a neighborhood of $y$ such that $\alpha\|x-y\|^2\le D_k(x,y)$ holds for any $x$ in the neighborhood.
\end{enumerate}
\end{assumption}

The local Lipschitz condition in Assumption \ref{assume:well-definedness} is a mild assumption on $f$.
In fact, any twice continuously differentiable function has locally Lipschitz gradient.
Owing to the lower semicontinuity of $g$, nonemptyness of $\mathcal{C}\cap\dom g$, and Assumption \ref{assume:well-definedness} (iv), it follows from Proposition \ref{prop:nonemptyness-compactness} that
\begin{equation}
    \emptyset\neq\argmin_{x\in\mathbb{E}}\left\{f(x^k)+\innerprod{\nabla f(x^k)}{x-x^k}+\beta^iD_k(x,x^k)+g(x)\right\}\subset\dom g
\end{equation}
for all $k\ge0$ and $i\ge0$ whenever $x^k\in\mathcal{C}$.
Therefore, for the well definedness of the subproblem \eqref{eq:general-prox-map}, we assume the following.

\begin{assumption}\label{assume:inclusion}
For any $k\ge0$ and $i\ge0$,
\begin{equation}
    \argmin_{x\in\mathbb{E}}\left\{f(x^k)+\innerprod{\nabla f(x^k)}{x-x^k}+\beta^iD_k(x,x^k)+g(x)\right\}\subset\mathcal{C}.
\end{equation}
\end{assumption}

Assumption \ref{assume:inclusion} is automatically valid when $\mathcal{C}=\mathbb{E}$.
When $\mathcal{C}\subsetneq\mathbb{E}$, Assumption \ref{assume:inclusion} can be satisfied by appropriately choosing the prox-grad distance (see Section \ref{subsec:IGM}).
Under Assumptions \ref{assume:well-definedness} and \ref{assume:inclusion}, the well definedness of Algorithm \ref{alg:GVDPGM} is established as follows.

\begin{lemma}
Suppose that Assumptions \ref{assume:well-definedness} and \ref{assume:inclusion} hold.
Then, the number of inner loop in Algorithm \ref{alg:GVDPGM} is finite in each iteration $k$.
\end{lemma}

\begin{proof}
Let $\{x^{k,i}\}$ be a sequence generated by \eqref{eq:general-prox-map}.
From the optimality of $x^{k,i}$, we have
\begin{equation}\label{eq:optimality-inner}
    \innerprod{\nabla f(x^k)}{x^{k,i}-x^k}+\beta^iD_k(x^{k,i},x^k)+g(x^{k,i})\le g(x^k).
\end{equation}
Using Assumption \ref{assume:well-definedness} (iv), we obtain
\begin{equation}\label{eq:linear-bound-inner}
    -\|\nabla f(x^k)\|\|x^{k,i}-x^k\|+(\beta^i-\gamma^*)D_k(x^{k,i},x^k)\le g(x^k)-l^*,
\end{equation}
where $\gamma^*\coloneqq(1+\gamma_{g,D_k})/2 \in (\gamma_{g,D_k}, 1)$ and $l^*\in\mathbb{R}$ is a lower bound of $g+\gamma^*D_k(\cdot,x^k)$.
Suppose that there exists a subsequence of $\{\|x^{k,i}-x^k\|\}$ converging to $c\in\mathbb{R}_{++}$, that is, $\|x^{k,i}-x^k\|\to_{I}c$, then $\{x^{k,i}\}_I$ is bounded.
Thus, since there is an accumulation point of $\{x^{k,i}\}_I$, which we denote by $x^*$, we see from the lower semicontinuity that $\liminf_{i\to_{I'}\infty}D_k(x^{k,i},x^k)\ge D_k(x^*,x^k)>0$ for some infinite set $I'\subset I$.
This and \eqref{eq:linear-bound-inner} imply
\begin{equation}
    -\|\nabla f(x^k)\|c+\infty\le g(x^k)-l^*,
\end{equation}
which is a contradiction.
On the other hand, if $\|x^{k,i}-x^k\|\to_{I}\infty$ for some infinite index set $I$, then a contradiction is derived from \eqref{eq:linear-bound-inner} and the supercoerciveness of $D_k(\cdot,x^k)+\delta_{\closure\mathcal{C}}$.
Consequently, $\{\|x^{k,i}-x^k\|\}$ converges to $0$.
From \eqref{eq:optimality-inner} and Assumptions \ref{assume:well-definedness} (i), (iii) and (v), it holds that
\begin{align}
    F(x^{k,i}) &=f(x^{k,i})+g(x^{k,i})\\
    &\le f(x^k)+\innerprod{\nabla f(x^k)}{x^{k,i}-x^k}+\frac{L_f}{2}\|x^{k,i}-x^k\|^2+g(x^{k,i})\\
    &\le f(x^k)+g(x^k)-\beta^iD_k(x^{k,i},x^k)+\frac{L_f}{2}\|x^{k,i}-x^k\|^2\\
    &\le F(x^k)-\left(1-\frac{L_f}{2\alpha\beta^i}\right)\beta^iD_k(x^{k,i},x^k)
\end{align}
for all sufficiently large $i$, where $L_f$ is the Lipschitz modulus of $\nabla f$ on a neighborhood of $x^k$.
Since $F_0=F(x^0)$ and it follows from the acceptance criterion \eqref{eq:acceptance-criterion} at the previous iteration that
\begin{equation}\label{eq:lower-bound}
    F_k\ge p_kF(x^k)+(1-p_k)\{F(x^k)+\sigma\beta^{i_{k-1}}D_{k-1}(x^k,x^{k-1})\}\ge F(x^k)
\end{equation}
for $k\ge1$, the acceptance criterion \eqref{eq:acceptance-criterion} holds for all sufficiently large $i$.
\end{proof}

The following properties hold for the GVDPGM.

\begin{proposition}\label{prop:property-GVDPGM}
Let $\{x^k\}$ be a sequence generated by Algorithm \ref{alg:GVDPGM} and suppose that Assumptions \ref{assume:well-definedness} and \ref{assume:inclusion} hold.
Then the following assertions hold:
\begin{enumerate}[(i)]
    \item The sequence $\{F_k\}$ is monotonically nonincreasing and bounded from below by $\inf_{x\in\mathbb{E}}F(x)$.
    In particular, it holds that
    \begin{equation}
        F(x^{k+1})\le F_{k+1}\le F_k-p\sigma\beta^{i_k}D_k(x^{k+1},x^k)
    \end{equation}
    for all $k\ge0$;
    \item The sequences $\{F_k\}$ and $\{F(x^k)\}$ converge to a same finite value;
    \item The sequence $\{x^k\}$ is included in the lower level set $\{x\in\mathbb{E}\mid F(x)\le F(x^0)\}\subset\dom F$;
    \item It holds that $\sum_{k=0}^\infty D_k(x^{k+1},x^k)\le\sum_{k=0}^\infty\beta^{i_k}D_k(x^{k+1},x^k)<\infty$, and hence $D_k(x^{k+1},x^k)\to0$ and $\beta^{i_k}D_k(x^{k+1},x^k)\to0$ hold.
\end{enumerate}
\end{proposition}

\begin{proof}
Using the acceptance criterion \eqref{eq:acceptance-criterion}, we have
\begin{equation}\label{eq:monotonicity-GVDPGM}
    F_{k+1}\le p_{k+1}\{F_k-\sigma\beta^{i_k}D_k(x^{k+1},x^k)\}+(1-p_{k+1})F_k\le F_k-p\sigma\beta^{i_k}D_k(x^{k+1},x^k).
\end{equation}
The lower bound is obtained as in \eqref{eq:lower-bound}.
Since the sequence $\{F_k\}$ is monotonically nonincreasing and bounded from below, $\{F_k\}$ converges to a finite value.
On the other hand, it follows from the definition and monotonicity of $F_k$ that
\begin{equation}
    F_k\ge F(x^k)=F_{k-1}+\frac{F_k-F_{k-1}}{p_k}\ge F_{k-1}+\frac{F_k-F_{k-1}}{p},
\end{equation}
which implies that $\{F(x^k)\}$ converges to the same limit as $\{F_k\}$.
As $\{F_k\}$ is monotonically nonincreasing, one has $F(x^k)\le F_k\le F_0=F(x^0)$.
Summing \eqref{eq:monotonicity-GVDPGM} from $k=0$ to $k'$ yields
\begin{equation}
    p\sigma\sum_{k=0}^{k'} D_k(x^{k+1},x^k)\le p\sigma\sum_{k=0}^{k'}\beta^{i_k}D_k(x^{k+1},x^k)\le F_0-F_{k'+1}\le F_0-\inf_{x\in\mathbb{E}}F(x)<\infty,
\end{equation}
which implies the last assertion.
\end{proof}

To prove subsequential convergence, we additionally make the following assumption.

\begin{assumption}\label{assume:local-error-bound}
For all $z\in\closure\mathcal{C}$, there exist a positive numbers $\alpha',~ \nu\le1$, and a neighborhood $\mathcal{N}_z$ of $z$ such that $\alpha'\|x-y\|^{1+\nu}\le D_k(x,y)$ holds for any $x\in\mathcal{C}$, $y\in\mathcal{N}_z\cap\mathcal{C}$, and $k\ge0$.
\end{assumption}

We note that Assumption \ref{assume:local-error-bound} implies Assumption \ref{assume:well-definedness} (v).
Lemma \ref{lem:boundedness-GVDPGM} plays a central role in our analysis.

\begin{lemma}\label{lem:boundedness-GVDPGM}
Let $\{x^k\}$ be a sequence generated by Algorithm \ref{alg:GVDPGM}.
Suppose that Assumptions \ref{assume:well-definedness} to \ref{assume:local-error-bound} hold.
Let $\{x^k\}_K$ be a subsequence of $\{x^k\}$ converging to some point $x^*$.
Then $\{\beta^{i_k}\}_K$ is bounded and $\|x^{k+1}-x^k\|\to_K0$.
\end{lemma}

\begin{proof}
It follows from Proposition \ref{prop:property-GVDPGM} (iii) and the lower semicontinuity of $F$ that $x^*\in\dom F\subset\closure\mathcal{C}$.
To derive a contradiction, we suppose that $\{\beta^{i_k}\}_K$ is unbounded.
Without loss of generality, we may assume that $\beta^{i_k}\to_K\infty$ and that $i_k\ge1$ holds for all $k\in K$, which implies that
\begin{equation}\label{eq:violation}
    F(\hat{x}^k)>F_k-\sigma\beta^{i_k-1}D_k(\hat{x}^k,x^k)\ge F(x^k)-\sigma\beta^{i_k-1}D_k(\hat{x}^k,x^k)
\end{equation}
where $\hat{x}^k\coloneqq x^{k,i_k-1}$ and the last inequality follows from Proposition \ref{prop:property-GVDPGM} (i).
Note that $\hat{x}^k\neq x^k$.
Using Proposition \ref{prop:property-GVDPGM} (iii) and \eqref{eq:linear-bound-inner} with $i=i_k-1$ yields
\begin{align}
    -\|\nabla f(x^k)\|\|\hat{x}^k-x^k\|+(\beta^{i_k-1}-\gamma^*)D_k(\hat{x}^k,x^k) &\le g(x^k)-l^*\\
    &= F(x^k)-f(x^k)-l^*\\
    &\le F(x^0)-f(x^k)-l^*.
\end{align}
As $x^k\to_Kx^*\in\closure\mathcal{C}$, we obtain from Assumption \ref{assume:local-error-bound} that
\begin{equation}\label{eq:linear-bound}
    -\|\nabla f(x^k)\|\|\hat{x}^k-x^k\|+\alpha'(\beta^{i_k-1}-\gamma^*)\|\hat{x}^k-x^k\|^{1+\nu}\le F(x^0)-f(x^k)-l^*
\end{equation}
for all sufficiently large $k\in K$.
If there exists an infinite set $K'\subset K$ such that $\|\hat{x}^k-x^k\|\to_{K'}c\in\mathbb{R}_{++}$, then we see from \eqref{eq:linear-bound} that
\begin{equation}
    -\|\nabla f(x^*)\|c+\infty\le F(x^0)-f(x^*)-l^*.
\end{equation}
On the other hand, if $\|\hat{x}^k-x^k\|\to_{K'}\infty$ for some infinite index set $K'\subset K$, then we see from \eqref{eq:linear-bound} that
\begin{equation}
    -\|\nabla f(x^*)\|+\infty\le 0.
\end{equation}
Consequently, we have $\|\hat{x}^k-x^k\|\to_K0$.
We note that $\hat{x}^k\to_Kx^*$.
From \eqref{eq:violation}, \eqref{eq:optimality-inner} with $i=i_k-1$, and Assumption \ref{assume:well-definedness} (iii), we obtain
\begin{align}
    \beta^{i_k-1}D_k(\hat{x}^k,x^k) &\le g(x^k)-g(\hat{x}^k)-\innerprod{\nabla f(x^k)}{\hat{x}^k-x^k}\\
    &=F(x^k)-F(\hat{x}^k)+f(\hat{x}^k)-f(x^k)-\innerprod{\nabla f(x^k)}{\hat{x}^k-x^k}\\
    &\le\sigma\beta^{i_k-1}D_k(\hat{x}^k,x^k)+\frac{L_f}{2}\|\hat{x}^k-x^k\|^2
\end{align}
for all sufficiently large $k\in K$, where $L_f$ is the Lipschitz modulus of $\nabla f$ on a neighborhood of $x^*$.
Combining this with Assumption \ref{assume:local-error-bound} yields
\begin{align}
    \alpha'(1-\sigma)\beta^{i_k-1}\|\hat{x}^k-x^k\|^{2} &\le\alpha'(1-\sigma)\beta^{i_k-1}\|\hat{x}^k-x^k\|^{1+\nu}\\
    &\le(1-\sigma)\beta^{i_k-1}D_k(\hat{x}^k,x^k)\\
    &\le\frac{L_f}{2}\|\hat{x}^k-x^k\|^2
\end{align}
for all sufficiently large $k\in K$, which contradicts to $\beta^{i_k-1}\to_K\infty$ because $\hat{x}^k\neq x^k$.
From Assumption \ref{assume:local-error-bound} and Proposition \ref{prop:property-GVDPGM} (iv), it holds that $\alpha'\|x^{k+1}-x^k\|^{1+\nu}\le D_k(x^{k+1},x^k)$ for all sufficiently large $k\in K$, and hence $\|x^{k+1}-x^k\|\to_K0$.
\end{proof}

For subsequential convergence to an F-stationary point, the following assumption is made.

\begin{assumption}\label{assume:inverse-local-error-bound}
For all $z\in\mathcal{C}$, there exist a positive numbers $L',~ \nu'$, and a neighborhood $\mathcal{N}_z'\subset\mathcal{C}$ of $z$ such that $D_k(x,y)\le L'\|x-y\|^{1+\nu'}$ holds for any $x,y\in\mathcal{N}_z'$ and $k\ge0$.
\end{assumption}

The subsequential convergence result for F-stationarity is obtained as follows.

\begin{theorem}\label{thm:global-convergence-F-stat}
Let $\{x^k\}$ be a sequence generated by Algorithm \ref{alg:GVDPGM}.
Suppose that Assumptions \ref{assume:well-definedness} to \ref{assume:inverse-local-error-bound} hold.
Then any accumulation point of $\{x^k\}$ contained in $\mathcal{C}$ is an F-stationary point of \eqref{problem:general}.
\end{theorem}

\begin{proof}
Let $\{x^k\}_K$ be a subsequence of $\{x^k\}$ converging to some point $x^*\in\mathcal{C}$.
We see from Lemma \ref{lem:boundedness-GVDPGM} that $\{\beta^{i_k}\}_K$ is bounded and $\{x^{k+1}\}_K$ also converges to $x^*$.
Note that $x^*\in\dom F$.
Let $d\in\mathcal{T}(x^*;\dom F)$ be fixed.
For any $(d',\eta)$ satisfying $\eta>0$ and $x^*+\eta d'\in\dom F$, from the optimality of $x^{k+1}$, we have
\begin{equation}
    \innerprod{\nabla f(x^k)}{x^{k+1}-x^*-\eta d'}+\beta^{i_k}D_k(x^{k+1},x^k)-\beta^{i_k}D_k(x^*+\eta d',x^k)+g(x^{k+1})\le g(x^*+\eta d').
\end{equation}
From this and Assumption \ref{assume:inverse-local-error-bound}, it holds that
\begin{equation}
    \innerprod{\nabla f(x^k)}{x^{k+1}-x^*-\eta d'}-\overline{\beta}L'\|x^*+\eta d'-x^k\|^{1+\nu'}+g(x^{k+1})\le g(x^*+\eta d')
\end{equation}
for all sufficiently large $k\in K$, sufficiently small $\eta>0$, and $d'$ sufficiently close to $d$, where $\overline{\beta}\coloneqq\sup_{k\in K}\beta^{i_k}<\infty$.
Combining this with the lower semicontinuity of $g$ and continuity of $\nabla f$ yields
\begin{equation}
    \eta\innerprod{\nabla f(x^*)}{d'}+g(x^*+\eta d')-g(x^*)+\eta^{1+\nu'}\overline{\beta}L'\|d'\|^{1+\nu'}\ge0.
\end{equation}
Dividing both sides by $\eta$ and taking the lower limit $(d',\eta)\to(d,0)$ give
\begin{align}
    F'(x^*;d)=\innerprod{\nabla f(x^*)}{d}+g'(x^*;d)\ge0,
\end{align}
which implies that $x^*$ is an F-stationary point.
\end{proof}

We make the following assumptions for subsequential convergence to an M-stationary point.

\begin{assumption}\label{assume:continuity-distance}
For any $x\in\mathcal{C}$, if $\{y^k\}\subset\mathcal{C}$ and $\{z^k\}\subset\mathcal{C}$ converge to $x$, then $D_k(y^k,z^k)\to0$.
\end{assumption}

\begin{assumption}\label{assume:differentiability-distance}
For all $k\ge0$ and $y\in\mathcal{C}$, the function $D_k(\cdot,y)$ is continuously differentiable on $\mathcal{C}$, denoted by $\nabla D_k(x,y)$ for its gradient at $x\in\mathcal{C}$.
\end{assumption}

\begin{assumption}\label{assume:continuity-gradient-distance}
For any $x\in\mathcal{C}$, if $\{y^k\}\subset\mathcal{C}$ and $\{z^k\}\subset\mathcal{C}$ converge to $x$, then $\nabla D_k(y^k,z^k)\to0$.
\end{assumption}

Note that Assumption \ref{assume:inverse-local-error-bound} implies Assumption \ref{assume:continuity-distance}.
From Assumption \ref{assume:continuity-distance}, given an infinite index set $K$, it is easy to see that the sequence $\{D_k(y^k,z^k)\}_K$ converges to $0$ for any subsequences $\{y^k\}_K$ and $\{z^k\}_K$ converging to $x\in\mathcal{C}$.
The same holds for Assumption \ref{assume:continuity-gradient-distance}.
Before we prove the subsequential convergence result for the M-stationarity, the convergence of the objective value is showed.

\begin{proposition}\label{prop:convergence-function-value}
Let $\{x^k\}$ be a sequence generated by Algorithm \ref{alg:GVDPGM}.
Suppose that Assumptions \ref{assume:well-definedness} to \ref{assume:local-error-bound} and \ref{assume:continuity-distance} hold.
Let $\{x^k\}_K$ be a subsequence of $\{x^k\}$ converging to some point $x^*\in\mathcal{C}$.
Then $\{F_k\}$ and $\{F(x^k)\}$ converge to $F(x^*)$.
\end{proposition}

\begin{proof}
We see from Lemma \ref{lem:boundedness-GVDPGM} that $\{\beta^{i_k}\}_K$ is bounded and $\{x^{k+1}\}_K$ also converges to $x^*$.
Since $x^{k+1}$ is optimal to the $k$th subproblem, we have
\begin{equation}
    \innerprod{\nabla f(x^k)}{x^{k+1}-x^k}+\beta^{i_k}D_k(x^{k+1},x^k)+g(x^{k+1})\le\innerprod{\nabla f(x^k)}{x^*-x^k}+\beta^{i_k}D_k(x^*,x^k)+g(x^*).
\end{equation}
From Assumption \ref{assume:continuity-distance} and the boundedness of $\{\beta^{i_k}\}_K$, $\beta^{i_k}D_k(x^{k+1},x^k)\to_K0$ and $\beta^{i_k}D_k(x^*,x^k)\to_K0$ hold, and hence taking the upper limit $k\to_K\infty$ gives
\begin{equation}
    \limsup_{k\to_K\infty}g(x^{k+1})\le g(x^*).
\end{equation}
Combining this with the lower semicontinuity of $g$ and continuity of $f$ on $\interior\dom f$ yields $F(x^{k+1})\to_KF(x^*)$.
In view of Proposition \ref{prop:property-GVDPGM} (ii), we have $\lim_{k\to\infty}F_k=\lim_{k\to\infty}F(x^k)=F(x^*)$.
\end{proof}

Using Proposition \ref{prop:convergence-function-value}, we have the following.

\begin{theorem}\label{thm:global-convergence-M-stat}
Let $\{x^k\}$ be a sequence generated by Algorithm \ref{alg:GVDPGM}.
Suppose that Assumptions \ref{assume:well-definedness} to \ref{assume:local-error-bound}, and \ref{assume:continuity-distance} to \ref{assume:continuity-gradient-distance} hold.
Then any accumulation point of $\{x^k\}$ contained in $\mathcal{C}$ is an M-stationary point of \eqref{problem:general}.
\end{theorem}

\begin{proof}
Let $\{x^k\}_K$ be a subsequence of $\{x^k\}$ converging to some point $x^*\in\mathcal{C}$.
We see from Lemma \ref{lem:boundedness-GVDPGM} that $\{\beta^{i_k}\}_K$ is bounded and $\{x^{k+1}\}_K$ also converges to $x^*$.
From the optimality of $x^{k+1}$, we have
\begin{equation}
    0\in\nabla f(x^k)+\beta^{i_k}\nabla D_k(x^{k+1},x^k)+\widehat{\partial}g(x^{k+1}),
\end{equation}
which implies
\begin{equation}
    \xi^k \coloneqq \nabla f(x^{k+1})-\nabla f(x^k)-\beta^{i_k}\nabla D_k(x^{k+1},x^k) \in \nabla f(x^{k+1}) + \widehat{\partial} g(x^{k+1}) = \widehat{\partial} F(x^{k+1}).
\end{equation}
We can see that $\xi^k\to_K0$ holds from the continuity of $\nabla f$ on $\interior\dom f$, Assumption \ref{assume:continuity-gradient-distance} and the boundedness of $\{\beta^{i_k}\}_K$.
Furthermore, Proposition \ref{prop:convergence-function-value} implies that $\lim_{k\to\infty}F(x^k)=F(x^*)$.
Thus, we have the desired result.
\end{proof}

Both Theorems \ref{thm:global-convergence-F-stat} and \ref{thm:global-convergence-M-stat} cannot account for accumulation points on the boundary of $\mathcal{C}$.
This issue will be addressed in Section \ref{subsec:IGM}.

\begin{remark}\label{remark:prox-Newton}
If $\{H_k\}$ is uniformly positive definite and bounded, which is a standard assumption for the global convergence of the proximal quasi-Newton-type method (see, e.g., \citep[Assumption 1]{tseng2009coordinate}), $D_k(x,y)=\frac{1}{2}\innerprod{x-y}{H_k(x-y)}$ satisfies Assumptions \ref{assume:local-error-bound} to \ref{assume:continuity-gradient-distance}.
Thus, the subsequential convergence result of the proximal quasi-Newton-type method for the fully nonconvex problem \eqref{problem:general} without global Lipschitz assumption is obtained as a corollary of our results.
Even when limited to the PGM, such convergence result for the F-stationarity is the first of its kind.
\end{remark}

\begin{remark}\label{remark:Bregman-PGM}
Let us consider $D_k(x,y)=D_h(x,y)=h(x)-h(y)-\innerprod{\nabla h(y)}{x-y}$ where $h:\mathbb{E}\to(-\infty,\infty]$ is a lower semicontinuous strictly convex function being continuously differentiable on $\mathcal{C}$.
\citet{bolte2018first} assume the strong convexity of $h$, the locally Lipschitz continuity of $\nabla f$ and $\nabla h$, and $\mathcal{C}=\mathbb{E}$ for the global convergence results.
Under such assumptions, it is easy to see that Assumptions \ref{assume:local-error-bound} to \ref{assume:continuity-gradient-distance} hold.
Accordingly, our results also provide a subsequential convergence result for the BPGM without the relative smoothness.
\end{remark}

\subsection{Global convergence and rate of convergence under KL assumption}
In the following, convergence of whole sequence and rate of convergence are established for the monotone case $p=1$, namely, $F_k=F(x^k)$ holds for all $k$.
To establish convergence results under the KL assumption, we first show the following lemma.

\begin{lemma}\label{lem:boundedness-on-ball}
Let $\{x^k\}$ be a sequence generated by Algorithm \ref{alg:GVDPGM}.
Suppose that Assumptions \ref{assume:well-definedness} to \ref{assume:local-error-bound} hold.
Then, for any $x^*\in\mathcal{C}$ and $\rho>0$ satisfying $\mathcal{B}_\rho(x^*)\subset\mathcal{C}$, there exists $\overline{\beta}_\rho>0$ such that $\beta^{i_k}\le\overline{\beta}_\rho$ holds for all $k$ with $x^k \in \mathcal{B}_\rho(x^*)$.
\end{lemma}

\begin{proof}
To derive contradiction, we suppose that such an upper bound $\overline{\beta}_\rho$ does not exist.
Let $K'$ be a subset of $\{k\ge0\mid x^k\in\mathcal{B}_\rho(x^*)\}$ satisfying $\beta^{i_k}\to_{K'}\infty$.
As the subsequence $\{x^k\}_{K'}$ is bounded because it is included in $\mathcal{B}_\rho(x^*)$, without loss of generality, we may assume that $\{x^k\}_{K'}$ converges to $\hat{x}\in\mathcal{B}_\rho(x^*)\subset\mathcal{C}$.
Lemma \ref{lem:boundedness-GVDPGM} implies the boundedness of $\{\beta^{i_k}\}_{K'}$, which is a contradiction.
\end{proof}

To provide the convergence analysis, we make an additional assumption.

\begin{assumption}
\label{assume:for-entire-convergence}
For all $z\in\mathcal{C}$, there exist a positive number $L''$ and a neighborhood $\mathcal{N}_z''\subset\mathcal{C}$ of $z$ such that $\|\nabla D_k(x,y)\|\le L''\|x-y\|$ holds for any $x,y\in\mathcal{N}_z''$, and $k\ge0$.
\end{assumption}

Note that Assumption \ref{assume:for-entire-convergence} implies Assumption \ref{assume:continuity-gradient-distance}.
For the proximal quasi-Newton-type method and the BPGM, under the same assumptions in Remarks \ref{remark:prox-Newton} and \ref{remark:Bregman-PGM}, Assumption \ref{assume:for-entire-convergence} is satisfied.
The following can be considered as a generalization of the main theorem by \citet{jia2023convergence}.

\begin{theorem}\label{thm:KL}
Let $\{x^k\}$ be a sequence generated by Algorithm \ref{alg:GVDPGM} with $p=1$.
Suppose that Assumptions \ref{assume:well-definedness} to \ref{assume:local-error-bound}, \ref{assume:continuity-distance}, \ref{assume:differentiability-distance}, and \ref{assume:for-entire-convergence} hold.
Let $\{x^k\}_K$ be a subsequence of $\{x^k\}$ converging to some point $x^*\in\mathcal{C}$ at which $F$ has the KL property\footnote{Note that $x^*\in\dom\partial F$ holds because we can apply Theorem \ref{thm:global-convergence-M-stat} to see $0\in\partial F(x^*)$.}.
Then $F(x^k)\to F(x^*)$ and $\sum_{k=0}^\infty\|x^{k+1}-x^k\|<\infty$ hold, particularly, $\{x^k\}$ also converges to $x^*$.
Moreover, if the corresponding desingularization function $\chi$ is of the form $\chi(t)=ct^\theta$ where $c>0$ and $0<\theta\le1$, then the following assertions hold:
\begin{enumerate}[(i)]
    \item If $\theta=1$, then $\{x^k\}$ converges in a finite number of steps;
    \item If $1/2<\theta<1$, then $\{F(x^k)\}$ and $\{x^k\}$ converges Q-superlinearly and R-superlinearly of order $\frac{1}{2(1-\theta)}$, respectively;
    \item If $\theta=1/2$, then $\{F(x^k)\}$ and $\{x^k\}$ converges Q-linearly and R-linearly, respectively;
    \item If $0<\theta<1/2$, then there exist $c_1, c_2>0$ such that
    \begin{align}
        F(x^k)-F(x^*)\le c_1k^{-\frac{1}{1-2\theta}},\\
        \|x^k-x^*\|\le c_2k^{-\frac{\theta}{1-2\theta}}.
    \end{align}
\end{enumerate}
\end{theorem}

\begin{proof}
Propositions \ref{prop:property-GVDPGM} (i) and \ref{prop:convergence-function-value} imply that $\{F(x^k)\}$ is monotonically nonincreasing and converging to $F(x^*)$.
Suppose that $F(x^k)=F(x^*)$ holds for some $k\ge0$.
Since it follows from the monotonicity that $F(x^{k+1})=F(x^*)$, by the acceptance criterion \eqref{eq:acceptance-criterion}, we have
\begin{equation}
    \sigma D_k(x^{k+1},x^k)\le F(x^k)-F(x^{k+1})=0
\end{equation}
and hence $x^{k+1}=x^k$.
Thus, $x^k=x^*$ holds for all sufficiently large $k$, which implies that all statements are valid.
Accordingly, for the remainder of the proof, we suppose that $F(x^k)>F(x^*)$ holds for all $k\ge0$.

We first define $\rho>0$ small enough to satisfy the following condition:
\begin{enumerate}[(a)]
    \item \label{proof:KL-ball} $\mathcal{B}_\rho(x^*)\subset\mathcal{U}\cap\mathcal{N}_{x^*}\cap\mathcal{N}_{x^*}''$;
\end{enumerate}
where $\mathcal{U},~ \mathcal{N}_{x^*}$, and $\mathcal{N}_{x^*}''$ are in Definition \ref{def:KL}, Assumption \ref{assume:local-error-bound}, and Assumption \ref{assume:for-entire-convergence}, respectively.
Note that $\mathcal{B}_\rho(x^*)\subset\mathcal{N}_{x^*}''\subset\mathcal{C}\subset\interior\dom f$.
In view of Proposition \ref{prop:property-GVDPGM} (iv) and Lemma \ref{lem:boundedness-GVDPGM},
we can take $k_0\in K$ large enough to satisfy the following:
\begin{enumerate}[(a)]
\setcounter{enumi}{1}
    \item \label{proof:KL-F} $F(x^*)<F(x^k)<F(x^*)+\varpi$ holds for all $k\ge k_0$;
    \item \label{proof:KL-D} $D_k(x^{k+1},x^k)\leq \alpha'$ for all $k \geq k_0$;
    \item \label{proof:KL-sum0} $\|x^{k_0}-x^*\|+2\|x^{k_0+1}-x^{k_0}\|+\frac{L_f+\overline{\beta}_\rho L''}{\sigma\alpha'}\chi(F(x^{k_0+1})-F(x^*))\le\rho$,
\end{enumerate}
where $\varpi,~ \alpha',~ L''$, and $\overline{\beta}_\rho$ are in Definition \ref{def:KL}, Assumption \ref{assume:local-error-bound}, Assumption \ref{assume:for-entire-convergence}, and Lemma \ref{lem:boundedness-on-ball}, $L_f$ is the Lipschitz modulus of $\nabla f$ on $\mathcal{B}_\rho(x^*)$, and $\chi$ is the desingularization function.
We now prove inductively that the following statements hold for all $k\ge k_0$:
\begin{enumerate}[(I)]
    \item \label{proof:KL-I} $x^k\in\mathcal{B}_\rho(x^*)$;
    \item \label{proof:KL-II} $\|x^{k_0}-x^*\|+\sum_{l=k_0}^k\|x^{l+1}-x^l\|\le\rho$.
\end{enumerate}
For $k=k_0$, it is obvious from condition (\ref{proof:KL-sum0}).
Assume that statements (\ref{proof:KL-I}) and (\ref{proof:KL-II}) hold for $k=k_0,\ldots,k'$.
By statement (\ref{proof:KL-II}) for $k=k'$, we obtain
\begin{equation}
    \|x^{k'+1}-x^*\|\le\|x^{k_0}-x^*\|+\sum_{l=k_0}^{k'}\|x^{l+1}-x^l\|\le\rho,
\end{equation}
namely, $x^{k'+1}\in\mathcal{B}_\rho(x^*)$ so that the statement (\ref{proof:KL-I}) holds for $k=k'+1$.
Note that combining (\ref{proof:KL-I}) with the condition (\ref{proof:KL-ball}) implies
\begin{equation}\label{eq:KL-xk-neigh}
x^k \in \mathcal{B}_\rho(x^*)\subset \mathcal{U} \cap \mathcal{N}_{x^*} \cap \mathcal{N}''_{x^*},\quad k=k_0,\ldots,k'+1.
\end{equation}
Owing to conditions (\ref{proof:KL-F}) and $x^k\in\mathcal{U}$, for $k=k_0+1,\ldots,k'+1$, we see from the KL property of $F$ that 
\begin{equation}\label{eq:KL-inequality}
    \chi'(F(x^k)-F(x^*))~\dist(0,\partial F(x^k))\ge1.
\end{equation}
From the optimality of $x^k$, it holds that
\begin{equation}
    0\in \nabla f(x^{k-1})+\beta^{i_{k-1}}\nabla D_{k-1}(x^k,x^{k-1})+\partial g(x^k),
\end{equation}
which is equivalent to
\begin{equation}
    \nabla f(x^k)-\nabla f(x^{k-1})-\beta^{i_{k-1}}\nabla D_{k-1}(x^k,x^{k-1})\in\partial F(x^k).
\end{equation}
Accordingly, by \eqref{eq:KL-xk-neigh}, Lemma \ref{lem:boundedness-on-ball}, Assumption \ref{assume:for-entire-convergence}, and the locally Lipschitz continuity of $\nabla f$, we have
\begin{align}
    \dist(0,\partial F(x^k)) &\le\|\nabla f(x^k)-\nabla f(x^{k-1})-\beta^{i_{k-1}}\nabla D_{k-1}(x^k,x^{k-1})\|\\
    &\le\|\nabla f(x^k)-\nabla f(x^{k-1})\|+\beta^{i_{k-1}}\|\nabla D_{k-1}(x^k,x^{k-1})\|\\
    &\le(L_f+\overline{\beta}_\rho L'')\|x^k-x^{k-1}\|
\end{align}
for $k=k_0+1,\ldots,k'+1$.
Combining this with \eqref{eq:KL-inequality} yields
\begin{equation}\label{eq:principal}
    \chi'(F(x^k)-F(x^*))\ge\frac{1}{(L_f+\overline{\beta}_\rho L'')\|x^k-x^{k-1}\|}
\end{equation}
for $k=k_0+1,\ldots,k'+1$.
Let $\Delta_{k,l}\coloneqq\chi(F(x^k)-F(x^*))-\chi(F(x^l)-F(x^*))$.
On the other hand, as we have \eqref{eq:KL-xk-neigh}, Assumption \ref{assume:local-error-bound} and the condition (\ref{proof:KL-D}),  it follows $\alpha'\|x^{k+1}-x^k\|^{1+\nu} \leq D_k(x^{k+1},x^k) \leq \alpha'$ and thus $\|x^{k+1}-x^k\| \leq 1$.
Hence, we obtain
\begin{equation}\label{eq:sufficient-decreasing}
    \sigma\alpha'\|x^{k+1}-x^k\|^2\le\sigma\alpha'\|x^{k+1}-x^k\|^{1+\nu}\le\sigma D_k(x^{k+1},x^k)\le F(x^k)-F(x^{k+1})
\end{equation}
for $k=k_0,\ldots,k'+1$, where the last inequality follows from the acceptance criterion \eqref{eq:acceptance-criterion}.
The concavity of $\chi$, \eqref{eq:principal}, and \eqref{eq:sufficient-decreasing} imply that
\begin{equation}
    \Delta_{k,k+1}\ge\chi'(F(x^k)-F(x^*))(F(x^k)-F(x^{k+1}))\ge\frac{\sigma\alpha'\|x^{k+1}-x^k\|^2}{(L_f+\overline{\beta}_\rho L'')\|x^k-x^{k-1}\|}
\end{equation}
for $k=k_0+1,\ldots,k'+1$.
Using relation $a+b\ge2\sqrt{ab}$, we obtain
\begin{equation}\label{eq:for-sum}
    c'\Delta_{k,k+1}+\|x^k-x^{k-1}\|\ge2\sqrt{\|x^{k+1}-x^k\|^2}=2\|x^{k+1}-x^k\|
\end{equation}
for $k=k_0+1,\ldots,k'+1$, where $c'=\frac{L_f+\overline{\beta}_\rho L''}{\sigma\alpha'}$.
By summing up, we have
\begin{align}
    2\sum_{k=k_0+1}^{k'+1}\|x^{k+1}-x^k\| &\le\sum_{k=k_0+1}^{k'+1}\|x^k-x^{k-1}\|+c'\Delta_{k_0+1,k'+2}\\
    &\le\sum_{k=k_0+1}^{k'+1}\|x^{k+1}-x^k\|+\|x^{k_0+1}-x^{k_0}\|+c'\chi(F(x^{k_0+1})-F(x^*)),
\end{align}
and hence it follows from the condition (\ref{proof:KL-sum0}) that
\begin{equation}
    \|x^{k_0}-x^*\|+\sum_{k=k_0}^{k'+1}\|x^{k+1}-x^k\|\le\|x^{k_0}-x^*\|+2\|x^{k_0+1}-x^{k_0}\|+c'\chi(F(x^{k_0+1})-F(x^*))\le\rho.
\end{equation}
This implies that statement (\ref{proof:KL-II}) holds for $k=k'+1$.
Consequently, statements (\ref{proof:KL-I}) and (\ref{proof:KL-II}) hold for all $k\ge k_0$.
From this, we immediately obtain that $\sum_{k=0}^\infty\|x^{k+1}-x^k\|<\infty$ and $x^k\to x^*$.

In the remaining proof, we suppose that $\chi(t)=ct^{\theta}$ for $c>0$  and $\theta \in (0,1]$.
If $\theta=1$, the inequality \eqref{eq:principal} implies
\begin{equation}
    \|x^{k+1}-x^k\|\ge\frac{1}{c(L_f+\overline{\beta}_\rho L'')}
\end{equation}
for all $k\ge k_0$, however, which contradicts to $\|x^{k+1}-x^k\|\to0$.
Thus, $\{x^k\}$ must converge to $x^*$ in a finite number of steps.

For $0<\theta<1$, inequalities \eqref{eq:principal} and \eqref{eq:sufficient-decreasing} yield
\begin{equation}\label{eq:difference-equation}
    R_k-R_{k+1}\ge\sigma\alpha'\|x^{k+1}-x^k\|^2\ge c''R_{k+1}^{2(1-\theta)}
\end{equation}
for all $k\ge k_0$, where $c''\coloneqq\frac{\sigma\alpha'}{(L_f+\overline{\beta}_\rho L'')^2c^2\theta^2}>0$ and $R_k\coloneqq F(x^k)-F(x^*)$.
On the other hand, from \eqref{eq:for-sum}, by summing from $k+1$ to $k'+1$, we obtain
\begin{equation}
    \sum_{l=k}^{k'+1}\|x^{l+1}-x^l\|\le2\|x^{k+1}-x^{k}\|+c'\chi(F(x^{k+1})-F(x^*))
\end{equation}
for any $k'\ge k\ge k_0$.
Since the left hand side is bounded from below by $\|x^k-x^{k'+2}\|$ which converges to $\|x^k-x^*\|$ as $k'\to \infty$,
combining this with \eqref{eq:sufficient-decreasing} and the monotonicity of $\{F(x^k)\}$ implies that
\begin{equation}\label{eq:sequence-bound-function}
    \|x^k-x^*\|\le\frac{2}{\sqrt{\sigma\alpha'}}(F(x^k)-F(x^{k+1}))^\frac{1}{2}+c'cR_k^\theta\le\frac{2}{\sqrt{\sigma\alpha'}}R_k^\frac{1}{2}+c'cR_k^\theta
\end{equation}
for all $k\ge k_0$.

If $1/2<\theta<1$, namely, $\frac{1}{2(1-\theta)}>1$, then we see from \eqref{eq:difference-equation} that
\begin{equation}
    c''R_{k+1}^{2(1-\theta)}\le R_k-R_{k+1}\le R_k,
\end{equation}
which is equivalent to
\begin{equation}
    R_{k+1}\le\left(\frac{1}{c''}R_k\right)^\frac{1}{2(1-\theta)}
\end{equation}
for all $k\ge k_0$, which implies Q-superlinear convergence of order $\frac{1}{2(1-\theta)}$ of $\{F(x^k)\}$.
By \eqref{eq:sequence-bound-function} and the fact that $R_k\to0$, there exists $c'''>0$ such that it holds that
\begin{equation}
    \|x^k-x^*\|\le c'''R_k^\frac{1}{2}
\end{equation}
for all sufficiently large $k$.
Since $\{R_k^\frac{1}{2}\}$ also converges Q-superlinearly of order $\frac{1}{2(1-\theta)}$, we obtain R-superlinear convergence of order $\frac{1}{2(1-\theta)}$ of $\{x^k\}$.

If $\theta=1/2$, the inequality \eqref{eq:difference-equation} yields
\begin{equation}
    R_{k+1}\le\frac{1}{1+c''}R_k
\end{equation}
for all $k\ge k_0$, which implies Q-linear convergence of $\{F(x^k)\}$.
By combining this with \eqref{eq:sequence-bound-function}, we have
\begin{equation}
    \|x^k-x^*\|\le\left(\frac{2}{\sqrt{\sigma\alpha'}}+c'c\right)R_k\le\left(\frac{2}{\sqrt{\sigma\alpha'}}+c'c\right)R_{k_0}\left(\frac{1}{1+c''}\right)^{k-k_0}
\end{equation}
for all $k\ge k_0$.
Thus, $\{x^k\}$ converges R-linearly.

If $0<\theta<1/2$, namely, $2(\theta-1)<-1$, then we see from \eqref{eq:difference-equation} that
\begin{equation}\label{eq:integral-bound}
    c''\left(\frac{R_k}{R_{k+1}}\right)^{2(\theta-1)}\le R_k^{2(\theta-1)}(R_k-R_{k+1})\le\int_{R_{k+1}}^{R_k}r^{2(\theta-1)}dr=\frac{R_k^{2\theta-1}}{2\theta-1}-\frac{R_{k+1}^{2\theta-1}}{2\theta-1}
\end{equation}
for all $k\ge k_0$.
On the other hand, it follows from Proposition \ref{prop:property-GVDPGM} (i) and $2\theta-1<0$ that
\begin{equation}\label{eq:monotone-bound}
    \frac{R_k^{2\theta-1}}{2\theta-1}-\frac{R_{k+1}^{2\theta-1}}{2\theta-1}=\frac{R_k^{2\theta-1}}{1-2\theta}\left\{\left(\frac{R_k}{R_{k+1}}\right)^{1-2\theta}-1\right\}\ge\frac{R_0^{2\theta-1}}{1-2\theta}\left\{\left(\frac{R_k}{R_{k+1}}\right)^{1-2\theta}-1\right\}
\end{equation}
for all $k\ge0$.
Combining \eqref{eq:integral-bound} and \eqref{eq:monotone-bound} yields
\begin{align}
    \frac{R_k^{2\theta-1}}{2\theta-1}-\frac{R_{k+1}^{2\theta-1}}{2\theta-1} &\ge\max\left\{c''\left(\frac{R_k}{R_{k+1}}\right)^{2(\theta-1)},c''''\left\{\left(\frac{R_k}{R_{k+1}}\right)^{1-2\theta}-1\right\}\right\}\\
    &\ge\min_{t\ge1}\max\{c''t^{2(\theta-1)},c''''(t^{1-2\theta}-1)\}\eqqcolon c_{\min}>0
\end{align}
for all $k\ge k_0$, where $c''''=R_0^{2\theta-1}/(2\theta-1)$.
By summing from $k_0$ to $k-1$, we have
\begin{equation}
    R_k^{2\theta-1}-R_{k_0}^{2\theta-1}\ge(1-2\theta)c_{\min}(k-k_0),
\end{equation}
which is equivalent to
\begin{equation}
    R_k\le\left\{R_{k_0}^{2\theta-1}+(1-2\theta)c_{\min}(k-k_0)\right\}^{-\frac{1}{1-2\theta}},
\end{equation}
and hence there exists $c_1>0$ such that
\begin{equation}
    R_k\le c_1k^{-\frac{1}{1-2\theta}}.
\end{equation}
In view of \eqref{eq:sequence-bound-function}, there is $c_2>0$ such that
\begin{equation}
    \|x^k-x^*\|\le c_2k^{-\frac{\theta}{1-2\theta}}.
\end{equation}
This completes the proof.
\end{proof}

Note that when $\frac{1}{2}<\theta<1$, the rate of convergence is often summarized as linear convergence similar to $\theta=\frac{1}{2}$, but in fact, it exhibits superlinear convergence.

\section{Applications}\label{sec:application}
In this section, two applications are presented to show the benefits of our results.
The first one is to robust logistic regression, where the smooth term is not Lipschitz continuous, and a prox-grad distance other than the Bregman divergence is used.
The second one is new convergence results for interior gradient methods in conic optimization, which, to the best of our knowledge, are the first rigorous results in the nonconvex setting.
In what follows, the inner product $\innerprod{\cdot}{\cdot}$ denotes the dot product for $\mathbb{R}^n$ and $\mathbb{R}^{m\times n}$.

\subsection{Exponential prox-grad distance for trimmed logistic regression}\label{subsec:TLR}
For a dataset $\{(b_j,a_j)\}_{j=1}^m\subset\{-1,1\}\times\mathbb{R}^p$, a robust estimation on binary regression problem is considered.
Let $\phi(a_j;x)$ represent a (generally nonlinear) regression model where $x\in\mathbb{R}^n$ denotes the model parameters.
Motivated by the least trimmed squares estimation by \citet{rousseeuw1984least}, we consider the following problem:
\begin{equation}\label{problem:TLR}
    \underset{x\in\mathbb{R}^n}{\mbox{minimize}} \quad T^\mathrm{logis}_K(b\circ\Phi(x))+R(x),
\end{equation}
where $b=(b_1,\ldots,b_m)^\top,~ \Phi(x)=(\phi(a_1;x),\ldots,\phi(a_m;x))^\top$, $R:\mathbb{R}^n\to(-\infty,\infty]$,
\begin{equation}\label{eq:trimmed-logistic}
    T^\mathrm{logis}_K(z)\coloneqq\min_{\substack{\Lambda\subset[m]\\|\Lambda|=m-K}}\sum_{j\in\Lambda}\log(1+e^{-z_j}),
\end{equation}
and $b\circ\Phi(x)$ denotes the Hadamard product of $b$ and $\Phi(x)$.
The first term of \eqref{problem:TLR} is the sum of the logistic loss functions for the remaining samples after removing the $K$ samples with the poorest fit under the given model.
On the other hand, $R(x)$ is a regularization term that induces prior knowledge into the model, such as the $\ell_1$ norm.
The case where $\phi(a_j;x)=\innerprod{a_j}{x}$ and $R(x)=\lambda\|x\|_1$ with $\lambda>0$ is considered by \citet{sun2020trimmed}.
In this case, \eqref{problem:TLR} is equivalent to
\begin{equation}
    \underset{x\in\mathbb{R}^n,|\Lambda|=m-K}{\mbox{minimize}} \quad \sum_{j\in\Lambda}\log(1+e^{-b_j\innerprod{a_j}{x}})+\lambda\|x\|_1,
\end{equation}
for which \citet{sun2020trimmed} proposed an alternating minimization algorithm with respect to $x$ and $\Lambda$.
Alternating minimization involving discrete variables such as $\Lambda$ tends to produce poor solutions empirically.

Here, we propose rewriting \eqref{problem:TLR} as an optimization problem that can be efficiently solved using the GVDPGM with a new prox-grad distance.
Define $l_\mathrm{LINEX}(\xi)\coloneqq e^\xi-\xi-1$ and 
\begin{equation}\label{eq:trimmed-exp}
    T^{\exp}_K(z)\coloneqq\min_{\substack{\Lambda\subset[m]\\|\Lambda|=m-K}}\sum_{j\in\Lambda}e^{-z_j}.
\end{equation}
The univariate function $l_\mathrm{LINEX}$ is referred to as the linear-exponential (LINEX) loss function \citep{varian1975bayesian,zellner1986bayesian}.
Then, the problem \eqref{problem:TLR} is equivalent to
\begin{equation}\label{problem:reformulated-TLR}
    \underset{x\in\mathbb{R}^n, z\in\mathbb{R}^m}{\mbox{minimize}} \quad \sum_{j=1}^ml_\mathrm{LINEX}(b_j\phi(a_j;x)-z_j)+T^{\exp}_K(z)+R(x).
\end{equation}

\begin{theorem}
If $(x^*,z^*)$ is a minimizer of \eqref{problem:reformulated-TLR}, then $x^*$ is optimal to \eqref{problem:TLR}.
Conversely, if $x^*$ is a minimizer of \eqref{problem:TLR}, then there exists $z^*\in\mathbb{R}^m$ such that $(x^*,z^*)$ is optimal to \eqref{problem:reformulated-TLR}.
\end{theorem}

\begin{proof}
Let $w=b\circ\Phi(x)$.
Then, we obtain
\begin{align}
    &\min_{z\in\mathbb{R}^m}\bigg\{\sum_{j=1}^m\left\{e^{w_j-z_j}-(w_j-z_j)-1\right\}+\min_{\substack{\Lambda\subset[m]\\|\Lambda|=m-K}}\sum_{j\in\Lambda}e^{-z_j}\bigg\}\\
    &=\min_{\substack{\Lambda\subset[m]\\|\Lambda|=m-K}}\min_{z\in\mathbb{R}^m}\bigg\{\sum_{j=1}^m\left\{e^{w_j-z_j}-(w_j-z_j)-1\right\}+\sum_{i\in\Lambda}e^{-z_j}\bigg\}\\
    &=\min_{\substack{\Lambda\subset[m]\\|\Lambda|=m-K}}\bigg\{\sum_{j\in\Lambda}\min_{z_j\in\mathbb{R}}\left\{e^{w_j-z_j}-(w_j-z_j)-1+e^{-z_j}\right\}+\sum_{j\notin\Lambda}\min_{z_j\in\mathbb{R}}\left\{e^{w_j-z_j}-(w_j-z_j)-1\right\}\bigg\}\\
    &=\min_{\substack{\Lambda\subset[m]\\|\Lambda|=m-K}}\bigg\{\sum_{j\in\Lambda}\log(1+e^{-w_j})\bigg\}=T^\mathrm{logis}_K(w)=T^\mathrm{logis}_K(b\circ\Phi(x)),
\end{align}
which completes the proof.
\end{proof}

This is a variant of the reformulation technique for the least trimmed squares proposed by \citet{yagishita2024fast}.
To construct the GVDPGM for \eqref{problem:reformulated-TLR}, we define the following exponential prox-grad distance:
\begin{equation}
    D_{\exp}(z,w) \coloneqq \sum_{j=1}^m\psi(z_j-w_j),
\end{equation}
where $\psi(\xi)\coloneqq(e^{\xi}+e^{-\xi})/2-1=\cosh(\xi)-1$.
It is easy to see that $D_{\exp}$ satisfies the conditions of Definition \ref{def:prox-distance}.
If the proximal mapping of $R$ has a closed form, then from the following proposition, the subproblem of the GVDPGM employing $D_{e,2}^{\gamma_1,\gamma_2}((x,z),(y,w))\coloneqq\gamma_1 D_{\exp}(z,w)+\gamma_2/2\|x-y\|^2$ with $\gamma_1,\gamma_2>0$ for \eqref{problem:reformulated-TLR} is explicitly computable.

\begin{proposition}
Let $a,w\in\mathbb{R}^m$, and $\gamma>0$.
We define $\mathcal{I}$ as the set consisting of index sets $I\subset[m]$ of size $m-K$ such that $\Psi(w_j,a_j)\le\Psi(w_{j'},a_{j'})$ for any $j\in I,~ j'\notin I$, where
\begin{equation}
    \Psi(w_j,a_j) \coloneqq a_j\log\frac{\sqrt{a_j^2+\gamma^2+2\gamma e^{-w_j}}-a_j}{\sqrt{a_j^2+\gamma^2}-a_j}+\sqrt{a_j^2+\gamma^2+2\gamma e^{-w_j}}-\sqrt{a_j^2+\gamma^2}.
\end{equation}
The set
\begin{equation}\label{eq:generalized-prox-map-trimmed-exp}
    \argmin_{z\in\mathbb{R}^m}\left\{\innerprod{a}{z}+\gamma D_{\exp}(z,w)+T^{\exp}_K(z)\right\}
\end{equation}
consists of a vector $z^*$ such that
\begin{align}
    z^*_j=
    \begin{cases}
        w_j-\log\gamma+\log(\sqrt{a_j^2+\gamma^2+2\gamma e^{-w_j}}-a_j), &j\in I,\\
        w_j-\log\gamma+\log(\sqrt{a_j^2+\gamma^2}-a_j), &j\notin I
    \end{cases}
\end{align}
for some $I\in\mathcal{I}$.
\end{proposition}

\begin{proof}
The minimization problem in \eqref{eq:generalized-prox-map-trimmed-exp} can be equivalently rewritten as
\begin{align}
    &\min_{z\in\mathbb{R}^m}\bigg\{\sum_{j=1}^m\big\{a_jz_j+\gamma\psi(z_j-w_j)\big\}+\min_{\substack{\Lambda\subset[m]\\|\Lambda|=m-K}}\sum_{j\in\Lambda}e^{-z_j}\bigg\}\\
    &=\min_{\substack{\Lambda\subset[m]\\|\Lambda|=m-K}}\min_{z\in\mathbb{R}^m}\bigg\{\sum_{j=1}^m\big\{a_jz_j+\gamma\psi(z_j-w_j)\big\}+\sum_{j\in\Lambda}e^{-z_j}\bigg\}\\
    &=\min_{\substack{\Lambda\subset[m]\\|\Lambda|=m-K}}\bigg\{\sum_{j\in\Lambda}\min_{z_j\in\mathbb{R}}\big\{a_jz_j+\gamma\psi(z_j-w_j)+e^{-z_j}\big\}+\sum_{j\notin\Lambda}\min_{z_j\in\mathbb{R}}\big\{a_jz_j+\gamma\psi(z_j-w_j)\big\}\bigg\}\\
    &=\min_{\substack{\Lambda\subset[m]\\|\Lambda|=m-K}}\bigg\{\sum_{j\in\Lambda}\big\{a_j(w_j-\log\gamma+\log(\sqrt{a_j^2+\gamma^2+2\gamma e^{-w_j}}-a_j))-\gamma+\sqrt{a_j^2+\gamma^2+2\gamma e^{-w_j}}\big\}\\
    &\hspace{6em}+\sum_{j\notin\Lambda}\big\{a_j(w_j-\log\gamma+\log(\sqrt{a_j^2+\gamma^2}-a_j))-\gamma+\sqrt{a_j^2+\gamma^2}\big\}\bigg\}\\
    &=\min_{\substack{\Lambda\subset[m]\\|\Lambda|=m-K}}\Bigg\{\sum_{j\in\Lambda}\bigg\{a_j\log\frac{\sqrt{a_j^2+\gamma^2+2\gamma e^{-w_j}}-a_j}{\sqrt{a_j^2+\gamma^2}-a_j}+\sqrt{a_j^2+\gamma^2+2\gamma e^{-w_j}}-\sqrt{a_j^2+\gamma^2}\bigg\}\\
    &\hspace{2em}+\sum_{j=1}^m\big\{a_j(w_j-\log\gamma+\log(\sqrt{a_j^2+\gamma^2}-a_j))-\gamma+\sqrt{a_j^2+\gamma^2}\big\}\Bigg\}\\
    &=\min_{\substack{\Lambda\subset[m]\\|\Lambda|=m-K}}\sum_{j\in\Lambda}\Psi(w_j,a_j)+\sum_{j=1}^m\big\{a_j(w_j-\log\gamma+\log(\sqrt{a_j^2+\gamma^2}-a_j))-\gamma+\sqrt{a_j^2+\gamma^2}\big\},
\end{align}
where the minimum value with respect to $z_j$ is attained at $z_j=w_j-\log\gamma+\log(\sqrt{a_j^2+\gamma^2+2\gamma e^{-w_j}}-a_j)$ for $j\in\Lambda$ and $z_j=w_j-\log\gamma+\log(\sqrt{a_j^2+\gamma^2}-a_j)$ for $j\notin\Lambda$.
This completes the proof.
\end{proof}

The resulting algorithm is novel, utilizing the prox-grad distance $D_{e,2}^{\gamma_1,\gamma_2}$ that is neither the Bregman divergence nor the $\phi$-divergence.
Under the assumptions that $R$ is prox-bounded and $\phi(a_j;\cdot)$ is sufficiently smooth such that the gradient of the first term of \eqref{problem:reformulated-TLR} is locally Lipschitz, as Assumptions \ref{assume:well-definedness} to \ref{assume:continuity-gradient-distance} are satisfied, Theorems \ref{thm:global-convergence-F-stat} and \ref{thm:global-convergence-M-stat} hold.
We only verify the validity of Assumptions \ref{assume:local-error-bound} and \ref{assume:inverse-local-error-bound} because the rest are immediate.

Since $\psi$ is $1$-strongly convex, we have
\begin{equation}
    \psi(\xi)\ge\psi(0)+\psi'(0)(\xi-0)+\frac{1}{2}(\xi-0)^2=\frac{1}{2}\xi^2.
\end{equation}
Thus, it holds that
\begin{equation}
    D_{e,2}^{\gamma_1,\gamma_2}((x,z),(y,w))\ge\frac{\gamma_1}{2}\|z-w\|^2+\frac{\gamma_2}{2}\|x-y\|^2\ge\frac{\min\{\gamma_1,\gamma_2\}}{2}\|(x,z)-(y,w)\|^2,
\end{equation}
which implies the fulfillment of Assumption \ref{assume:local-error-bound}.

Let $\tilde{L}>1$.
From simple calculations, we see that the solutions to the equation $\tilde{L}-\psi''(\xi)=0$ are $\xi^*_\pm\coloneqq\pm\log(\tilde{L}+\sqrt{\tilde{L}^2-1})$.
Since $\tilde{L}-\psi''(\xi)$ is nonnegative on $[0,\xi^*_+]$, it holds that $\tilde{L}\xi-\psi'(\xi)\ge-\psi'(0)=0$ for all $\xi\in[0,\xi^*_+]$.
Again, the nonnegativity of $\tilde{L}\xi-\psi'(\xi)$ implies the nonnegativity of $\tilde{L}\xi^2/2-\psi(\xi)$ on $[0,\xi^*_+]$.
By the symmetry of $\tilde{L}\xi^2/2-\psi(\xi)$, it is nonnegative for all $\xi\in[-\xi^*_+,\xi^*_+]$.
Thus, it holds that
\begin{align}
    D_{e,2}^{\gamma_1,\gamma_2}((x,z),(y,w))\le\frac{\tilde{L}\gamma_1}{2}\|z-w\|^2+\frac{\gamma_2}{2}\|x-y\|^2\le\frac{\max\{\tilde{L}\gamma_1,\gamma_2\}}{2}\|(x,z)-(y,w)\|^2
\end{align}
whenever $\|(x,z)-(y,w)\|\le\xi^*_+$, which implies that Assumption \ref{assume:inverse-local-error-bound} is satisfied.

Finally, we discuss conditions to ensure Theorem \ref{thm:KL} when $\phi(a_j;x)=\innerprod{a_j}{x}$.
The next lemma gives a sufficient condition for which the objective function of \eqref{problem:reformulated-TLR} enjoys KL property with some exponent, relying on the subanaliticity (cf. \cite{shiota1997,bolte2007KL}).

\begin{lemma}\label{lem:KL-exponent}
If $\phi(a_j;x)=\innerprod{a_j}{x}$ holds and if
$R(x)$ is a subanalytic function such that $\inf R \in \mathbb{R}$, $\dom R$ is closed, and $R|_{\dom R}$ is continuous, then the objective function of \eqref{problem:reformulated-TLR} has the KL property at any $(x^*,z^*) \in \dom\partial R \times \mathbb{R}^m$ with some exponent $\theta \in (0,1]$.
\end{lemma}

\begin{proof}
The objective function of \eqref{problem:reformulated-TLR} can be rewritten as
\begin{equation}\label{eq:min-representation-trimmed-logistic}
    \min_{\substack{\Lambda\subset[m]\\|\Lambda|=m-K}}\underbrace{\sum_{j=1}^ml_\mathrm{LINEX}(b_j\innerprod{a_j}{x}-z_j)+\sum_{j\in\Lambda}e^{-z_j}+R(x)}_{F_\Lambda(x,z)},
\end{equation}
when $\phi(a_j;x)=\innerprod{a_j}{x}$.
We firstly prove that $F_\Lambda$ is subanalytic.
It is easy to see that $l_\mathrm{LINEX}(\xi)=e^\xi-\xi-1$ is a real analytic, convex, and nonnegative function.
Thus, so is the function $(x,z)\mapsto\sum_{j=1}^ml_\mathrm{LINEX}(b_j\innerprod{a_j}{x}-z_j)+\sum_{j\in\Lambda}e^{-z_j}$ from which, in particular, it is subanalytic and bounded from below.
Hence, $F_\Lambda(x,z)$ is subanalytic by the fact that the sum of two subanalytic functions is subanalytic when these two are bounded from below \cite{shiota1997}. 

To prove the assertion, take any $(x^*,z^*) \in \dom\partial R\times \mathbb{R}^m = \dom\partial F_\Lambda$.
Since $F_\Lambda$ is a continuous subanalytic function with $\dom F_\Lambda = \dom R \times \mathbb{R}^m$ being closed, applying \cite[Theorem 3.1]{bolte2007KL} shows that $F_\Lambda(x,z)$ has KL property at $(x^*,z^*)$ with some exponent $\theta_\Lambda \in (0,1]$.
As the function \eqref{eq:min-representation-trimmed-logistic} is the minimum of finitely many $F_\Lambda$'s, it follows from \cite[Theorem 3.1]{li2017KL} that \eqref{eq:min-representation-trimmed-logistic} has KL property at $(x^*,z^*)$ with (at least) the exponent $\theta\coloneqq \min\{\theta_\Lambda:\Lambda \subset [m],~|\Lambda|=m-K\}$.
\end{proof}

Since we have $|\psi'(\xi)| = |\psi'(\xi)-\psi'(0)| \leq \max_{\zeta \in [-1,1]} |\psi''(\zeta)| \cdot |\xi-0| = \cosh(1)|\xi|$ for all $\xi \in [-1,1]$, one can see that Assumption \ref{assume:for-entire-convergence} is satisfied for $D_{e,2}^{\gamma_1,\gamma_2}$.
Consequently, under the conditions in Lemma \ref{lem:KL-exponent}, Theorem \ref{thm:KL} applies to the problem  \eqref{eq:min-representation-trimmed-logistic} with the GVDPGM employing the prox-grad distance $D_{e,2}^{\gamma_1,\gamma_2}$.
Furthermore, if $R$ is convex and coercive, then \eqref{eq:min-representation-trimmed-logistic} is also coercive and is the pointwise minimum of finitely many convex functions, and hence it follows from \citep[Corollary 1.6]{yagishita2025exact} that the generated sequence converges to a local minimum of \eqref{eq:min-representation-trimmed-logistic}.

\subsection{Interior gradient method for conic optimization}\label{subsec:IGM}
When $\mathcal{C}\subsetneq\mathbb{E}$, the sequence generated by the GVDPGM remains in $\mathcal{C}$, that is, the interior of $\closure\mathcal{C}$.
Gradient methods that remains in the interior in this manner are called interior gradient methods \citep{eggermont1990multiplicative,iusem1995interior,iusem1996multiplicative,auslender2004interior,auslender2006interior,bauschke2017descent,takahashi2026majorization,hua2016block}.
If $f$ is differentiable on $\mathcal{C}$ but has nondifferentiable points on its boundary, the projected gradient method cannot be used because the gradient information may not be accessible at the projected points.
However, in such cases, the interior gradient methods remain applicable since the sequence generated by them stays within the interior on which $f$ is differentiable.

Theorems \ref{thm:global-convergence-F-stat} and \ref{thm:global-convergence-M-stat} do not guarantee the stationarity when the accumulation point is on the boundary of $\mathcal{C}$.
Below, choosing appropriate prox-grad distances, we introduce interior gradient methods that allow access to a closed-form solution of the subproblem and that ensure convergence even on the boundary.

\subsubsection{Nonnegative orthant}
Let $\mathbb{R}_+^n$ be the nonnegative orthant and $\mathbb{R}_{++}^n$ be the interior of $\mathbb{R}_+^n$.
We consider the following constrained optimization problem:
\begin{equation}\label{problem:nonnegative-cone}
    \underset{x\in\mathbb{R}^n_+}{\mbox{minimize}} \quad f(x)+\tilde{g}(x),
\end{equation}
where $f$ is continuously differentiable.
By setting $g=\tilde{g}+\delta_{\mathbb{R}_+^n}$ using the indicator function of the nonnegative orthant $\delta_{\mathbb{R}_+^n}$, \eqref{problem:nonnegative-cone} can be regarded as a special case of \eqref{problem:general}.
We consider using the prox-grad distance with $\mathcal{C}=\mathbb{R}_{++}^n$.
The following two examples are optimization problems for which the interior gradient methods are effective.

\begin{example}[Poisson linear inverse problem \citep{shepp1982maximum,vardi1985statistical}]
Let $b\in\mathbb{R}_{++}^m$ and $a_j\in\mathbb{R}_+^n$ for $j=1,\ldots,m$.
Suppose that $a_j\neq0$ for all $j=1,\ldots,m$.
The Poisson linear inverse problem is formulated as
\begin{equation}\label{problem:PLIP}
    \underset{x\in\mathbb{R}_+^n}{\mbox{minimize}} \quad \underbrace{\sum_{j=1}^m\{\innerprod{a_j}{x}-b_j\log\innerprod{a_j}{x}\}}_{f(x)}+\tilde{g}(x).
\end{equation}
As $\innerprod{a_j}{x}>0$ holds for all $j$ whenever $x\in\mathbb{R}_{++}^n$, we see that
\begin{equation}
    \dom f=\interior\dom f=\{x\in\mathbb{R}^m\mid\innerprod{a_j}{x}>0~ \mbox{for all}~ j\}\supset\mathbb{R}_{++}^n,
\end{equation}
and hence $f$ is differentiable on $\mathbb{R}_{++}^n$.
On the other hand, $f$ is not differentiable at $0\in\mathbb{R}_+^n$.
\end{example}

\begin{example}[KL-NMF \citep{lee2000algorithms}]
Let $V\in\mathbb{R}_+^{m\times n}$.
Suppose that $V\neq0$.
The nonnegative matrix factorization with the Kullback--Leibler divergence (KL-NMF) is formulated as
\begin{equation}\label{problem:KL-NMF}
    \underset{(W,H)\in\mathbb{R}_+^{m\times r}\times\mathbb{R}_+^{r\times n}}{\mbox{minimize}} \quad \underbrace{\sum_{j=1}^m\sum_{j'=1}^n\{(WH)_{jj'}-V_{jj'}\log(WH)_{jj'}\}}_{f(W,H)}+\tilde{g}(W,H).
\end{equation}
The interior $\mathbb{R}_{++}^{m\times r}\times\mathbb{R}_{++}^{r\times n}$ is included in
\begin{equation}
    \dom f=\interior\dom f=\{(W,H)\in\mathbb{R}^{m\times r}\times\mathbb{R}^{r\times n}\mid(WH)_{jj'}>0~ \mbox{for all}~ (j,j')~ \mbox{satisfying}~ V_{jj'}>0\}.
\end{equation}
The function $f$ is differentiable on the interior, but is not differentiable at $(0,0)\in\mathbb{R}_+^{m\times r}\times\mathbb{R}_+^{r\times n}$.
\end{example}

While the interior gradient methods for smooth convex problems proposed by \citet{eggermont1990multiplicative,iusem1995interior}, and \citet{iusem1996multiplicative} do not require the global descent lemma for $f$, they are impractical because they require either knowledge of the local Lipschitz constant or an exact line search to determine the stepsize at each iteration.
The interior gradient methods of \citet{auslender2004interior,auslender2006interior,bauschke2017descent}, and \citet{takahashi2026majorization} coincide with the GVDPGM using a prox-grad distance of the form
\begin{equation}\label{eq:distance-nonnegative-cone}
    D_{\mathbb{R}_{++}^n}^{\gamma_1,\gamma_2}(x,y)\coloneqq\gamma_1\sum_{j=1}^n(y_j)^r\left(-\log\frac{x_j}{y_j}+\frac{x_j}{y_j}-1\right)+\frac{\gamma_2}{2}\|x-y\|^2
    \in [0,\infty]
\end{equation}
for $x\in\mathbb{R}^n, y\in\mathbb{R}_{++}^n$, where $\gamma_1,\gamma_2>0$, and $r\ge0$ (the choice of stepsize differs in each case).
When $r=0$, it is called the regularized Burg divergence \citep{bolte2003barrier,attouch2004regularized,auslender2006interior,bauschke2017descent,takahashi2026majorization}, which is the Bregman divergence generated by a kernel $h(x)=-\gamma_1\sum_{j=1}^n\log x_j+\frac{\gamma_2}{2}\|x\|^2$; when $r=1$, it is a regularized $\phi$-divergence \citep{bolte2003barrier,attouch2004regularized,auslender2006interior}; and when $r=2$, it is called the logarithmic-quadratic proximal distance \citep{auslender1999logarithmic,auslender1999interior,bolte2003barrier,auslender2004interior,auslender2006interior}.
For smooth problems (i.e., $\tilde{g}=0$), the solution of subproblem \eqref{eq:general-prox-map} with $D_{\mathbb{R}_{++}^n}^{\gamma_1,\gamma_2}$ is given by
\begin{equation}
    x^{k,i}_j=\frac{-(\nabla_j f(x^k)+\tilde{\gamma_1}(x^k_j)^{r-1}-\tilde{\gamma_2}x^k_j)+\sqrt{(\nabla_j f(x^k)+\tilde{\gamma_1}(x^k_j)^{r-1}-\tilde{\gamma_2}x^k_j)^2+4\tilde{\gamma_1}\tilde{\gamma_2}(x^k_j)^r}}{2\tilde{\gamma_2}},
\end{equation}
where $\tilde{\gamma_1}=\beta^i\gamma_1, \tilde{\gamma_2}=\beta^i\gamma_2$, and $\nabla_j f=\partial f/\partial x_j$ (see \citet{auslender2004interior,auslender2006interior}).
\citet{auslender2004interior,auslender2006interior} provided convergence analyses of their interior gradient method for smooth convex problems under the assumption of Lipschitz continuity of the gradient.
For composite problems, convergence results of the interior gradient methods are given by \citet{bauschke2017descent} and \citet{takahashi2026majorization} under the assumption of relative smoothness.
Note that only \citet{takahashi2026majorization} addressed nonconvex problems; in all other works, even nonconvexity of the smooth term was not allowed.
Although \citet{takahashi2026majorization} have showed that the smooth term of the KL-NMF is smooth relative to $h(x)=-\gamma_1\sum_{j=1}^n\log x_j+\frac{\gamma_2}{2}\|x\|^2$, this does not hold for most of nonconvex functions.
Moreover, their analysis also assumes the level-boundedness of $F$, which is not satisfied in pure KL-NMF (i.e., $g=0$).
Additionally, in fact, stationarity has not been established when the accumulation point is on the boundary.
Here, a convergence analysis of interior gradient methods for nonconvex composite problems that also accounts for the boundary is provided.
\citet{takahashi2026majorization} used the regularized Burg divergence (i.e., \eqref{eq:distance-nonnegative-cone} with $r=0$), but we show below that the GVDPGM using \eqref{eq:distance-nonnegative-cone} with $r>1$ can handle convergence on the boundary.

\begin{theorem}\label{thm:nonnegative-cone}
Let $\{\gamma_{1,k}\}_{k=0}^\infty, \{\gamma_{2,k}\}_{k=0}^\infty\subset[\gamma_{\min},\gamma_{\max}]$ with $\gamma_{\max}\ge\gamma_{\min}>0$.
Suppose that Assumptions \ref{assume:well-definedness} (i)--(iv) and \ref{assume:inclusion} hold, $\dom g=\mathbb{R}_+^n$, and $g$ is continuous on $\mathbb{R}_+^n$.
Let $\{x^k\}$ be a sequence generated by Algorithm \ref{alg:GVDPGM} with $D_k=D_{\mathbb{R}_{++}^n}^{\gamma_{1,k},\gamma_{2,k}}$ and $r>1$.
Then any accumulation point of $\{x^k\}$ is an F-stationary point of \eqref{problem:general}.
\end{theorem}

\begin{proof}
As we see that
\begin{equation}
    D_k(x,y)=D_{\mathbb{R}_{++}^n}^{\gamma_{1,k},\gamma_{2,k}}(x,y)\ge\frac{\gamma_{\min}}{2}\|x-y\|^2
\end{equation}
for any $x\in\mathbb{R}^n, y\in\mathbb{R}_{++}^n$, Assumption \ref{assume:local-error-bound} is satisfied (and thus, Assumption \ref{assume:well-definedness} (v) is as well).
On the other hand, for any $z\in\mathbb{R}_{++}^n$, there exists a bounded neighborhood $\mathcal{N}_z'\subset\mathbb{R}_{++}^n$ of $z$ and a positive constant $\tilde{L}$ such that
\begin{equation}
    \sum_{j=1}^n\left(-\log\frac{x_j}{y_j}+\frac{x_j}{y_j}-1\right)\le\frac{\tilde{L}}{2}\|x-y\|^2
\end{equation}
for all $x,y\in\mathcal{N}_z'$, because the gradient of $x\mapsto-\sum_{j=1}^n\log x_j$ is locally Lipschitz continuous on $\mathbb{R}_{++}^n$.
Accordingly, it holds that
\begin{align}
    D_{\mathbb{R}_{++}^n}^{\gamma_{1,k},\gamma_{2,k}}(x,y) &\le\gamma_{\max}\sum_{j=1}^n\left(\sup_{y\in\mathcal{N}_z'}\max_j(y_j)^r\right)\left(-\log\frac{x_j}{y_j}+\frac{x_j}{y_j}-1\right)+\frac{\gamma_{\max}}{2}\|x-y\|^2\\
    &\le\frac{\gamma_{\max}}{2}\left(1+\tilde{L}\sup_{y\in\mathcal{N}_z'}\max_j(y_j)^r\right)\|x-y\|^2,
\end{align}
which implies Assumption \ref{assume:inverse-local-error-bound}.

Let $\{x^k\}_K$ be a subsequence of $\{x^k\}$ converging to some point $x^*$.
It follows from Proposition \ref{prop:property-GVDPGM} (iii) and the lower semicontinuity of $F$ that $x^*\in\dom F\subset\mathbb{R}_+^n$.
If $x^*\in\mathbb{R}_{++}^n$ holds, the stationarity of $x^*$ is established from Theorem \ref{thm:global-convergence-F-stat}.

The remaining part of this proof assumes that $x^*\in\mathbb{R}_+^n\setminus\mathbb{R}_{++}^n$.
We define $\{\delta^k\}\subset\mathbb{R}_+^n$ as
\begin{align}
    \delta_j^k\coloneqq
    \begin{cases}
        x_j^k, &x_j^*=0,\\
        0, &\mbox{otherwise},
    \end{cases}
\end{align}
then $\delta^k\to_K0$.
Note that $\{\beta^{i_k}\}_K$ is bounded and $\|x^{k+1}-x^k\|\to_K0$ by Lemma \ref{lem:boundedness-GVDPGM}.
Let $d\in\mathcal{T}(x^*;\dom F)$ be fixed.
For any $(d',\eta)$ satisfying $\eta>0$ and $x^*+\eta d'\in\dom F\subset\mathbb{R}_+^n$, from the optimality of $x^{k+1}$ to the subproblem, we have
\begin{align}
    &\innerprod{\nabla f(x^k)}{x^{k+1}}+g(x^{k+1})\\
    &\le\innerprod{\nabla f(x^k)}{x^{k+1}}+\beta^{i_k}D_k(x^{k+1},x^k)+g(x^{k+1})\\
    &\le\innerprod{\nabla f(x^k)}{x^*+\eta d'+\delta^k}+\beta^{i_k}D_k(x^*+\eta d'+\delta^k,x^k)+g(x^*+\eta d'+\delta^k).
\end{align}
We note that $x^*+\eta d'+\delta^k\in\mathbb{R}_{++}^n$ holds for all sufficiently small $\eta>0$ and $d'$ sufficiently close to $d$.
Rearranging the above yields
\begin{align}\label{eq:optimality-nn-cone}
\begin{split}
    &\innerprod{\nabla f(x^k)}{x^{k+1}-x^*-\eta d'-\delta^k}+g(x^{k+1})-g(x^*+\eta d'+\delta^k)\\
    &\le\overline{\beta}\gamma_{\max}\sum_{j=1}^n(x_j^k)^r\left(-\log\frac{x_j^*+\eta d'_j+\delta_j^k}{x_j^k}+\frac{x_j^*+\eta d'_j+\delta_j^k}{x_j^k}-1\right)+\frac{\overline{\beta}\gamma_{\max}}{2}\|x^*+\eta d'+\delta^k-x^k\|^2,
\end{split}
\end{align}
where $\overline{\beta}\coloneqq\sup_{k\in K}\beta^{i_k}<\infty$.
For $j$ satisfying $x_j^*=0$, since $x_j^k\to_Kx_j^*=0$, we have
\begin{align}
    &(x_j^k)^r\left(-\log\frac{x_j^*+\eta d'_j+\delta_j^k}{x_j^k}+\frac{x_j^*+\eta d'_j+\delta_j^k}{x_j^k}-1\right)\\
    &=-(x_j^k)^r\log(\eta d'_j+x_j^k)+(x_j^k)^r\log x_j^k+(x_j^k)^{r-1}(\eta d'_j+x_j^k)-(x_j^k)^r\to_K0
\end{align}
There exists $\tilde{L}>0$ such that it holds that
\begin{align}
    &(x_j^k)^r\left(-\log\frac{x_j^*+\eta d'_j+\delta_j^k}{x_j^k}+\frac{x_j^*+\eta d'_j+\delta_j^k}{x_j^k}-1\right)\\
    &\le\frac{\tilde{L}}{2}(x_j^*+\eta d'_j-x_j^k)^2
\end{align}
for any $j$ satisfying $x_j^*\neq0$, sufficiently small $\eta>0$, $d'$ sufficiently close to $d$, and sufficiently large $k\in K$.
Consequently, taking the limit $k\to_K\infty$ of \eqref{eq:optimality-nn-cone} yields
\begin{equation}
    \innerprod{\nabla f(x^*)}{-\eta d'}+g(x^*)-g(x^*+\eta d')\le\frac{\overline{\beta}\gamma_{\max}(1+\tilde{L})}{2}\|\eta d'\|^2
\end{equation}
because $g$ is continuous on $\mathbb{R}_+^n$ and $\nabla f$ is continuous on $\dom F\subset\interior\dom f$.
Rearranging this, dividing both sides by $\eta$, and taking the lower limit $(d',\eta)\to(d,0)$ give
\begin{align}
    F'(x^*;d)=\innerprod{\nabla f(x^*)}{d}+g'(x^*;d)\ge0,
\end{align}
which implies that $x^*$ is an F-stationary point.
\end{proof}

Theorem \ref{thm:nonnegative-cone} can be considered the first flawless convergence result of the interior gradient method for nonconvex composite problems, which does not require any global descent lemma and level-boundedness.
Since $r>1$, the above setting lies beyond the BPGM framework.
As it is generally difficult to establish conditions analogous to the relative smoothness, we emphasize that it is appropriate not to assume any global descent lemma.
Note that the analyses by \citet{hua2016block} and \citet{bonettini2016variable} are not applicable to the GVDPGM with \eqref{eq:distance-nonnegative-cone}, even when $\tilde{g}$ is convex, due to their assumptions on the distance.

Finally, we discuss the subproblem in the case where $\tilde{g}\neq0$.
\citet{takahashi2026majorization} have derived closed form solutions of the subproblem with the $\ell_1$ regularization $\tilde{g}(W,H)=\lambda_1\|W\|_1+\lambda_2\|H\|_1$ and the Tikhonov regularization $\tilde{g}(W,H)=\lambda_1\|W\|^2+\lambda_2\|H\|^2$.
However, since $\lambda_1\|W\|_1+\lambda_2\|H\|_1=\lambda_1\sum_{j,j'}W_{jj'}+\lambda_2\sum_{j,j'}H_{jj'}$ holds on $\mathbb{R}_+^{m\times r}\times\mathbb{R}_+^{r\times n}$ and these functions are smooth, they can be included in \( f \), eliminating the need to consider the composite form.
Here, we consider a nonconvex nonsmooth function for which the subproblem solution can be easily computed.
The trimmed $\ell_1$ norm is defined by
\begin{equation}
    T_K(x)\coloneqq\min_{\substack{\Lambda\subset[n]\\|\Lambda|=n-K}}\sum_{j\in\Lambda}|x_j|
\end{equation}
for $x\in\mathbb{R}^n$, which is a nonconvex nonsmooth function introduced by \citet{luo2013new} and \citet{huang2015two} to obtain a more clear-cut sparse solution than the $\ell_1$ norm.
The trimmed $\ell_1$ norm is known as an exact penalty function of the cardinality constraint \citep{ahn2017difference,gotoh2018dc,amir2021trimmed,lu2023exact,yagishita2025exact}:
\begin{equation}
    \|x\|_0\coloneqq|\{j\mid x_j\neq0\}|\le K.
\end{equation}
When $\tilde{g}$ is chosen as the indicator function of the cardinality constraint, the assumption on the domain of $\tilde{g}$ in Theorem \ref{thm:nonnegative-cone} is not satisfied; however, this assumption is satisfied when $T_K$ is used as a penalty function.
Although the trimmed $\ell_1$ norm is not only nonconvex and nonsmooth but also nonseparable, the solution of the subproblem can be computed as follows.

\begin{proposition}
Let $a\in\mathbb{R}^n, y\in\mathbb{R}_{++}^n, \gamma_1,\gamma_2>0$, and $\lambda>0$.
We define $\mathcal{I}$ as the set consisting of index sets $I\subset[n]$ of size $n-K$ such that $\Psi(y_j,a_j)\le\Psi(y_{j'},a_{j'})$ for any $j\in I,~ j'\notin I$, where
\begin{align}
    \Psi(y_j,a_j) &\coloneqq\Upsilon(\chi(y_j,a_j+\lambda),y_j,a_j+\lambda)-\Upsilon(\chi(y_j,a_j),y_j,a_j),\\
    \Upsilon(\xi,\upsilon,\alpha) &\coloneqq(\alpha+\gamma_1\upsilon^{r-1}-\gamma_2\upsilon)\xi+\frac{\gamma_2}{2}\xi^2-\gamma_1\upsilon^r\log\xi,\\
    \chi(\upsilon,\alpha) &\coloneqq\frac{-(\alpha+\gamma_1\upsilon^{r-1}-\gamma_2\upsilon)+\sqrt{(\alpha+\gamma_1\upsilon^{r-1}-\gamma_2\upsilon)^2+4\gamma_1\gamma_2\upsilon^r}}{2\gamma_2}.
\end{align}
The set
\begin{equation}\label{eq:generalized-prox-map-trimmed}
    \argmin_{x\in\mathbb{R}^n}\left\{\innerprod{a}{x}+D_{\mathbb{R}_{++}^n}^{\gamma_1,\gamma_2}(x,y)+\lambda T_K(x)\right\}
\end{equation}
consists of a vector $x^*$ such that
\begin{align}
    x^*_j=
    \begin{cases}
        \chi(y_j,a_j+\lambda), &j\in I,\\
        \chi(y_j,a_j), &j\notin I
    \end{cases}
\end{align}
for some $I\in\mathcal{I}$.
\end{proposition}

\begin{proof}
It is easy to see that $\chi(\upsilon,\alpha)$ is a unique optimal solution of $\min_{\xi>0}\Upsilon(\xi,\upsilon,\alpha)$ for $\upsilon>0,~ \alpha\in\mathbb{R}$.
The minimization problem in \eqref{eq:generalized-prox-map-trimmed} can be equivalently rewritten as
\begin{align}
    &\min_{x\in\mathbb{R}_{++}^n}\bigg\{\sum_{j=1}^n\big\{(a_j+\gamma_1(y_j)^{r-1}-\gamma_2y_j)x_j+\frac{\gamma_2}{2}(x_j)^2-\gamma_1(y_j)^r\log x_j\big\}+\lambda\min_{\substack{\Lambda\subset[n]\\|\Lambda|=n-K}}\sum_{j\in\Lambda}x_j\bigg\}\\
    &=\min_{\substack{\Lambda\subset[n]\\|\Lambda|=n-K}}\min_{x\in\mathbb{R}_{++}^n}\bigg\{\sum_{j=1}^n\big\{(a_j+\gamma_1(y_j)^{r-1}-\gamma_2y_j)x_j+\frac{\gamma_2}{2}(x_j)^2-\gamma_1(y_j)^r\log x_j\big\}\big\}+\lambda\sum_{j\in\Lambda}x_j\bigg\}\\
    &=\min_{\substack{\Lambda\subset[n]\\|\Lambda|=n-K}}\bigg\{\sum_{j\in\Lambda}\min_{x_j>0}\Upsilon(x_j,y_j,a_j+\lambda)+\sum_{j\notin\Lambda}\min_{x_j>0}\Upsilon(x_j,y_j,a_j)\bigg\}\\
    &=\min_{\substack{\Lambda\subset[n]\\|\Lambda|=n-K}}\bigg\{\sum_{j\in\Lambda}\Upsilon(\chi(y_j,a_j+\lambda),y_j,a_j+\lambda)+\sum_{j\notin\Lambda}\Upsilon(\chi(y_j,a_j),y_j,a_j)\bigg\}\\
    &=\min_{\substack{\Lambda\subset[n]\\|\Lambda|=n-K}}\bigg\{\sum_{j\in\Lambda}\big\{\Upsilon(\chi(y_j,a_j+\lambda),y_j,a_j+\lambda)-\Upsilon(\chi(y_j,a_j),y_j,a_j)\big\}+\sum_{j=1}^n\Upsilon(\chi(y_j,a_j),y_j,a_j)\bigg\}\\
    &=\min_{\substack{\Lambda\subset[n]\\|\Lambda|=n-K}}\sum_{j\in\Lambda}\Psi(a_j,w_j)+\sum_{j=1}^n\Upsilon(\chi(y_j,a_j),y_j,a_j).
\end{align}
This completes the proof.
\end{proof}

\subsubsection{Positive semidefinite cone and second-order cone}
Let $\mathbb{S}_+^n$ be the positive semidefinite cone and $\mathbb{L}_+^n$ be the second order cone, namely, $\mathbb{S}^n_+=\{X\in\mathbb{S}^n\mid X\succeq0\}$ and $\mathbb{L}_+^n=\{x\in\mathbb{R}^n\mid x_n\ge(\sum_{j=1}^{n-1}x_j^2)^{1/2}\}$, where $\mathbb{S}^n$ is the set of symmetric $n\times n$ matrices.
The interiors of $\mathbb{S}_+^n$ and $\mathbb{L}_+^n$ are denoted by $\mathbb{S}_{++}^n$ and $\mathbb{L}_{++}^n$, respectively.
We consider the following constrained optimization problems:
\begin{align}
    \underset{X\in\mathbb{S}^n_+}{\mbox{minimize}} \quad f(X),\label{problem:PSD-cone}\\
    \underset{x\in\mathbb{L}^n_+}{\mbox{minimize}} \quad f(x).\label{problem:SO-cone}
\end{align}
where $f$ is continuously differentiable.
For these problems, the use of
\begin{align}
    D_{\mathbb{S}_{++}^n}^{\gamma_1,\gamma_2}(X,Y)\coloneqq\gamma_1\det(Y)^r\left(-\log\frac{\det(X)}{\det(Y)}+\innerprod{X}{Y^{-1}}-n\right)+\frac{\gamma_2}{2}\|X-Y\|^2,\label{eq:distance-PSD-cone}\\
    D_{\mathbb{L}_{++}^n}^{\gamma_1,\gamma_2}(x,y)\coloneqq\gamma_1\innerprod{y}{Jy}^r\left(-\log\frac{\innerprod{x}{Jx}}{\innerprod{y}{Jy}}+2\frac{\innerprod{x}{Jy}}{\innerprod{y}{Jy}}-2\right)+\frac{\gamma_2}{2}\|x-y\|^2,\label{eq:distance-SO-cone}
\end{align}
are considered, where $X\in\mathbb{S}^n, Y\in\mathbb{S}_{++}^n, x\in\mathbb{R}^n, y\in\mathbb{L}_{++}^n, \gamma_1, \gamma_2>0, r\ge0$, and $J$ is a diagonal matrix with its first $n-1$ entries being $-1$ and the last entry being $1$.
In this case, the corresponding subproblems of the GVDPGM have closed form solutions that are contained in the interiors (see \citet{auslender2006interior} for details).
\citet{auslender2006interior} showed convergence results of the interior gradient methods using \eqref{eq:distance-PSD-cone} and \eqref{eq:distance-SO-cone} with $r=0$ for the above conic optimization problems under the assumptions of the convexity of $f$ and the Lipschitz continuity of $\nabla f$.
Here, the convergence of the interior gradient methods using \eqref{eq:distance-PSD-cone} and \eqref{eq:distance-SO-cone} with $r>1$ is established without assuming either the convexity or the Lipschitz continuity.
The proofs follow the same approach as in the case of the nonnegative orthant, but for completeness, they are provided in the appendix.

\begin{theorem}\label{thm:PSD-cone}
Let $\{\gamma_{1,k}\}_{k=0}^\infty, \{\gamma_{2,k}\}_{k=0}^\infty\subset[\gamma_{\min},\gamma_{\max}]$ with $\gamma_{\max}\ge\gamma_{\min}>0$.
Assume the following:
\begin{itemize}
    \item The function $f:\mathbb{S}^n\to(-\infty,\infty]$ is continuously differentiable on $\interior\dom f$, bounded from below on $\mathbb{S}_+^n$, and lower semicontinuous;
    \item It holds that $\mathbb{S}_{++}^n\subset\interior\dom f$ and $\dom f\cap\mathbb{S}_+^n\subset\interior\dom f$;
    \item The function $\nabla f$ is locally Lipschitz continuous on $\interior\dom f$.
\end{itemize}
Let $\{X^k\}$ be a sequence generated by Algorithm \ref{alg:GVDPGM} with $D_k=D_{\mathbb{S}_{++}^n}^{\gamma_{1,k},\gamma_{2,k}}$ and $r>1$ for \eqref{problem:PSD-cone}.
Then any accumulation point of $\{X^k\}$ is an F-stationary point of \eqref{problem:PSD-cone}.
\end{theorem}

\begin{theorem}\label{thm:SO-cone}
Let $\{\gamma_{1,k}\}_{k=0}^\infty, \{\gamma_{2,k}\}_{k=0}^\infty\subset[\gamma_{\min},\gamma_{\max}]$ with $\gamma_{\max}\ge\gamma_{\min}>0$.
Assume the following:
\begin{itemize}
    \item The function $f:\mathbb{R}^n\to(-\infty,\infty]$ is continuously differentiable on $\interior\dom f$, bounded from below on $\mathbb{L}_+^n$, and lower semicontinuous;
    \item It holds that $\mathbb{L}_{++}^n\subset\interior\dom f$ and $\dom f\cap\mathbb{L}_+^n\subset\interior\dom f$;
    \item The function $\nabla f$ is locally Lipschitz continuous on $\interior\dom f$.
\end{itemize}
Let $\{x^k\}$ be a sequence generated by Algorithm \ref{alg:GVDPGM} with $D_k=D_{\mathbb{L}_{++}^n}^{\gamma_{1,k},\gamma_{2,k}}$ and $r>1$ for \eqref{problem:SO-cone}.
Then any accumulation point of $\{x^k\}$ is an F-stationary point of \eqref{problem:SO-cone}.
\end{theorem}

\section{Conclusion}\label{sec:conclusion}
In the paper, we have conducted convergence analyses of proximal gradient-type methods without global descent lemma.
Thanks to our results, it become possible to apply a new proximal gradient-type method with a general proximal term to the trimmed logistic regression problem.
As a byproduct, new convergence results of the interior gradient methods for conic optimization are also provided.
We hope that our results will lead to the development of new proximal gradient-type methods.

We have conducted the convergence analysis in the presence of the KL property only for the case of the monotone Armijo condition ($p=1$).
By utilizing the techniques from a very recent paper by \citet{kanzow2025convergence}, it may be possible to derive a similar result for the nonmonotone case as well.
When the interior gradient methods are considered, our analysis under the KL assumption (Theorem \ref{thm:KL}) cannot account for the boundary.
Solving this issue is an important challenge.

\section*{Acknowledgments}
The first author was supported partly by the JSPS Grant-in-Aid for Early-Career Scientists 25K21158.
The second author was supported partly by the JSPS Grant-in-Aid for Early-Career Scientists 25K15010.
The authors would like to thank Akifumi Okuno, Hironori Fujisawa, and Keisuke Yano for their helpful comments.

\appendix
\def\thesection{Appendix \Alph{section}}

\section{Proofs}\label{sec:proofs}
In this section, we provide the proofs omitted in the main body of this paper.

\begin{proof}[Proof of Theorem \ref{thm:PSD-cone}]
From the assumptions of Theorem \ref{thm:PSD-cone}, setting $g=\delta_{\mathbb{S}_+^n}$ and $\mathcal{C}=\mathbb{S}_{++}^n$, Assumption \ref{assume:well-definedness} (i)--(iv) holds.
As we see that
\begin{equation}
    D_k(X,Y)=D_{\mathbb{S}_{++}^n}^{\gamma_{1,k},\gamma_{2,k}}(X,Y)\ge\frac{\gamma_{\min}}{2}\|X-Y\|^2
\end{equation}
for any $X\in\mathbb{S}^n, Y\in\mathbb{S}_{++}^n$, Assumption \ref{assume:local-error-bound} is satisfied (and thus, Assumption \ref{assume:well-definedness} (v) is as well).
On the other hand, for any $Z\in\mathbb{S}_{++}^n$, there exists a bounded neighborhood $\mathcal{N}_Z'\subset\mathbb{S}_{++}^n$ of $Z$ and a positive constant $\tilde{L}$ such that
\begin{equation}
    -\log\frac{\det(X)}{\det(Y)}+\innerprod{X}{Y^{-1}}-n=-\log\frac{\det(X)}{\det(Y)}+\innerprod{Y^{-1}}{X-Y}\le\frac{\tilde{L}}{2}\|X-Y\|^2
\end{equation}
for all $X,Y\in\mathcal{N}_Z'$, because the gradient of $X\mapsto-\log\det(X)$ is locally Lipschitz continuous on $\mathbb{S}_{++}^n$.
Accordingly, it holds that
\begin{align}
    D_{\mathbb{S}_{++}^n}^{\gamma_{1,k},\gamma_{2,k}}(X,Y) &\le\gamma_{\max}\left(\sup_{Y\in\mathcal{N}_Z'}\det(Y)^r\right)\left(-\log\frac{\det(X)}{\det(Y)}+\innerprod{X}{Y^{-1}}-n\right)+\frac{\gamma_{\max}}{2}\|X-Y\|^2\\
    &\le\frac{\gamma_{\max}}{2}\left(1+\tilde{L}\sup_{Y\in\mathcal{N}_Z'}\det(Y)^r\right)\|X-Y\|^2,
\end{align}
which implies Assumption \ref{assume:inverse-local-error-bound}.

Let $\{X^k\}_K$ be a subsequence of $\{X^k\}$ converging to some point $X^*$.
It follows from Proposition \ref{prop:property-GVDPGM} (iii) and the lower semicontinuity of $f+\delta_{\mathbb{S}_+^n}$ that $X^*\in\dom f\cap\mathbb{S}_+^n\subset\mathbb{S}_+^n$.
If $X^*\in\mathbb{S}_{++}^n$ holds, the stationarity of $X^*$ is established from Theorem \ref{thm:global-convergence-F-stat}.

The remaining part of this proof assumes that $X^*\in\mathbb{S}_+^n\setminus\mathbb{S}_{++}^n$.
Note that $\det(X^k)\to_K0$.
We define $\{\Delta^k\}\subset\mathbb{S}_{++}^n$ as $\Delta^k\coloneqq\lambda_{\min}(X^k)I$, where $\lambda_{\min}(X^k)$ is the smallest eigenvalue of $X^k$ and $I$ is the identity matrix, then $\Delta^k\to_K0$.
Note that $\{\beta^{i_k}\}_K$ is bounded and $\|X^{k+1}-X^k\|\to_K0$ by Lemma \ref{lem:boundedness-GVDPGM}.
Let $D\in\mathcal{T}(X^*;\dom f\cap\mathbb{S}_+^n)$ be fixed.
For any $(D',\eta)$ satisfying $\eta>0$ and $X^*+\eta D'\in\dom f\cap\mathbb{S}_+^n\subset\mathbb{S}_+^n$, from the optimality of $X^{k+1}$ to the subproblem, we have
\begin{align}
    \innerprod{\nabla f(X^k)}{X^{k+1}} &\le\innerprod{\nabla f(X^k)}{X^{k+1}}+\beta^{i_k}D_k(X^{k+1},X^k)\\
    &\le\innerprod{\nabla f(X^k)}{X^*+\eta D'+\Delta^k}+\beta^{i_k}D_k(X^*+\eta D'+\Delta^k,X^k).
\end{align}
We note that $X^*+\eta D'+\Delta^k\in\mathbb{S}_{++}^n$ holds.
Rearranging the above yields
\begin{align}\label{eq:optimality-PSD-cone}
\begin{split}
    &\innerprod{\nabla f(X^k)}{X^{k+1}-X^*-\eta D'-\Delta^k}\\
    &\le\overline{\beta}\gamma_{\max}\det(X^k)^r\left(-\log\frac{\det(X^*+\eta D'+\Delta^k)}{\det(X^k)}+\innerprod{X^*+\eta D'+\Delta^k}{(X^k)^{-1}}-n\right)\\
    &\quad+\frac{\overline{\beta}\gamma_{\max}}{2}\|X^*+\eta D'+\Delta^k-X^k\|^2,
\end{split}
\end{align}
where $\overline{\beta}\coloneqq\sup_{k\in K}\beta^{i_k}<\infty$.
It holds that
\begin{align}
    &\det(X^k)^r\left(-\log\frac{\det(X^*+\eta D'+\Delta^k)}{\det(X^k)}+\innerprod{X^*+\eta D'+\Delta^k}{(X^k)^{-1}}-n\right)\\
    &\le-\det(X^k)^r\log\det(X^*+\eta D'+\Delta^k)+\det(X^k)^r\log\det(X^k)\\
    &\quad+\det(X^k)^r\|X^*+\eta D'+\Delta^k\|\|(X^k)^{-1}\|-n\det(X^k)^r\\
    &\le-\det(X^k)^r\log\det(X^*+\eta D'+\Delta^k)+\det(X^k)^r\log\det(X^k)\\
    &\quad+\det(X^k)^r\|X^*+\eta D'+\Delta^k\|\frac{\sqrt{n}}{\lambda_{\min}(X^k)}-n\det(X^k)^r\\
    &=-\det(X^k)^r\left(\sum_{j=1}^n\log(\lambda_j+\lambda_{\min}(X^k))\right)+\det(X^k)^r\log\det(X^k)\\
    &\quad+\det(X^k)^r\|X^*+\eta D'+\Delta^k\|\frac{\sqrt{n}}{\lambda_{\min}(X^k)}-n\det(X^k)^r\\
    &\to_K0,
\end{align}
where $\lambda_j$ is the $j$th largest eigenvalue of $X^*+\eta D'$.
Consequently, taking the limit $k\to_K\infty$ of \eqref{eq:optimality-PSD-cone} yields
\begin{equation}
    \innerprod{\nabla f(X^*)}{-\eta D'}\le\frac{\overline{\beta}\gamma_{\max}}{2}\|\eta D'\|^2
\end{equation}
because $\nabla f$ is continuous on $\dom f\cap\mathbb{S}_+^n\subset\interior\dom f$.
Rearranging this, dividing both sides by $\eta$, and taking the limit $(D',\eta)\to(D,0)$ give
\begin{align}
    f'(X^*;D)=\innerprod{\nabla f(X^*)}{D}\ge0,
\end{align}
which implies that $X^*$ is an F-stationary point.
\end{proof}

\begin{proof}[Proof of Theorem \ref{thm:SO-cone}]
From the assumptions of Theorem \ref{thm:SO-cone}, setting $g=\delta_{\mathbb{L}_+^n}$ and $\mathcal{C}=\mathbb{L}_{++}^n$, Assumption \ref{assume:well-definedness} (i)--(iv) holds.
As we see that
\begin{equation}
    D_k(x,y)=D_{\mathbb{L}_{++}^n}^{\gamma_{1,k},\gamma_{2,k}}(x,y)\ge\frac{\gamma_{\min}}{2}\|x-y\|^2
\end{equation}
for any $x\in\mathbb{R}^n, y\in\mathbb{L}_{++}^n$, Assumption \ref{assume:local-error-bound} is satisfied (and thus, Assumption \ref{assume:well-definedness} (v) is as well).
On the other hand, for any $z\in\mathbb{L}_{++}^n$, there exists a bounded neighborhood $\mathcal{N}_z'\subset\mathbb{L}_{++}^n$ of $z$ and a positive constant $\tilde{L}$ such that
\begin{equation}
    -\log\frac{\innerprod{x}{Jx}}{\innerprod{y}{Jy}}+2\frac{\innerprod{x}{Jy}}{\innerprod{y}{Jy}}-2=-\log\frac{\innerprod{x}{Jx}}{\innerprod{y}{Jy}}+2\frac{\innerprod{Jy}{x-y}}{\innerprod{y}{Jy}}\le\frac{\tilde{L}}{2}\|x-y\|^2
\end{equation}
for all $x,y\in\mathcal{N}_z'$, because the gradient of $x\mapsto-\log\innerprod{x}{Jx}$ is locally Lipschitz continuous on $\mathbb{L}_{++}^n$.
Accordingly, it holds that
\begin{align}
    D_{\mathbb{L}_{++}^n}^{\gamma_{1,k},\gamma_{2,k}}(x,y) &\le\gamma_{\max}\left(\sup_{y\in\mathcal{N}_z'}\innerprod{y}{Jy}^r\right)\left(-\log\frac{\innerprod{x}{Jx}}{\innerprod{y}{Jy}}+2\frac{\innerprod{x}{Jy}}{\innerprod{y}{Jy}}-2\right)+\frac{\gamma_{\max}}{2}\|x-y\|^2\\
    &\le\frac{\gamma_{\max}}{2}\left(1+\tilde{L}\sup_{y\in\mathcal{N}_z'}\innerprod{y}{Jy}^r\right)\|x-y\|^2,
\end{align}
which implies Assumption \ref{assume:inverse-local-error-bound}.

Let $\{x^k\}_K$ be a subsequence of $\{x^k\}$ converging to some point $x^*$.
It follows from Proposition \ref{prop:property-GVDPGM} (iii) and the lower semicontinuity of $f+\delta_{\mathbb{L}_+^n}$ that $x^*\in\dom f\cap\mathbb{L}_+^n\subset\mathbb{L}_+^n$.
If $x^*\in\mathbb{L}_{++}^n$ holds, the stationarity of $x^*$ is established from Theorem \ref{thm:global-convergence-F-stat}.

The remaining part of this proof assumes that $x^*\in\mathbb{L}_+^n\setminus\mathbb{L}_{++}^n$.
Note that $\innerprod{x^k}{Jx^k}\to_K0$.
We define $\{\delta^k\}\subset\mathbb{L}_{++}^n$ as $\delta^k\coloneqq(0,\ldots,0,\innerprod{x^k}{Jx^k})^\top$, then $\delta^k\to_K0$.
Note that $\{\beta^{i_k}\}_K$ is bounded and $\|x^{k+1}-x^k\|\to_K0$ by Lemma \ref{lem:boundedness-GVDPGM}.
Let $d\in\mathcal{T}(x^*;\dom f\cap\mathbb{L}_+^n)$ be fixed.
For any $(d',\eta)$ satisfying $\eta>0$ and $x^*+\eta d'\in\dom f\cap\mathbb{L}_+^n\subset\mathbb{L}_+^n$, from the optimality of $x^{k+1}$ to the subproblem, we have
\begin{align}
    \innerprod{\nabla f(x^k)}{x^{k+1}} &\le\innerprod{\nabla f(x^k)}{x^{k+1}}+\beta^{i_k}D_k(x^{k+1},x^k)\\
    &\le\innerprod{\nabla f(x^k)}{x^*+\eta d'+\delta^k}+\beta^{i_k}D_k(x^*+\eta d'+\delta^k,x^k).
\end{align}
We note that $x^*+\eta d'+\delta^k\in\mathbb{L}_{++}^n$ holds.
Rearranging the above yields
\begin{align}\label{eq:optimality-SO-cone}
\begin{split}
    &\innerprod{\nabla f(x^k)}{x^{k+1}-x^*-\eta d'-\delta^k}\\
    &\le\overline{\beta}\gamma_{\max}\innerprod{x^k}{Jx^k}^r\left(-\log\frac{\innerprod{x^*+\eta d'+\delta^k}{J(x^*+\eta d'+\delta^k)}}{\innerprod{x^k}{Jx^k}}+2\frac{\innerprod{x^*+\eta d'+\delta^k}{Jx^k}}{\innerprod{x^k}{Jx^k}}-2\right)\\
    &\quad+\frac{\overline{\beta}\gamma_{\max}}{2}\|x^*+\eta d'+\delta^k-x^k\|^2,
\end{split}
\end{align}
where $\overline{\beta}\coloneqq\sup_{k\in K}\beta^{i_k}<\infty$.
It holds that
\begin{align}
    &\innerprod{x^k}{Jx^k}^r\left(-\log\frac{\innerprod{x^*+\eta d'+\delta^k}{J(x^*+\eta d'+\delta^k)}}{\innerprod{x^k}{Jx^k}}+2\frac{\innerprod{x^*+\eta d'+\delta^k}{Jx^k}}{\innerprod{x^k}{Jx^k}}-2\right)\\
    &=-\innerprod{x^k}{Jx^k}^r\log\frac{\innerprod{x^*+\eta d'}{J(x^*+\eta d')}+2(x^*_n+\eta d'_n)\innerprod{x^k}{Jx^k}+\innerprod{x^k}{Jx^k}^2}{\innerprod{x^k}{Jx^k}}\\
    &\quad+2\innerprod{x^k}{Jx^k}^{r-1}\innerprod{x^*+\eta d'+\delta^k}{Jx^k}-2\innerprod{x^k}{Jx^k}^r\\
    &\to_K0.
\end{align}
Consequently, taking the limit $k\to_K\infty$ of \eqref{eq:optimality-SO-cone} yields
\begin{equation}
    \innerprod{\nabla f(x^*)}{-\eta d'}\le\frac{\overline{\beta}\gamma_{\max}}{2}\|\eta d'\|^2
\end{equation}
because $\nabla f$ is continuous on $\dom f\cap\mathbb{L}_+^n\subset\interior\dom f$.
Rearranging this, dividing both sides by $\eta$, and taking the limit $(d',\eta)\to(d,0)$ give
\begin{align}
    f'(x^*;d)=\innerprod{\nabla f(x^*)}{d}\ge0,
\end{align}
which implies that $x^*$ is an F-stationary point.
\end{proof}

\bibliography{reference.bib}
\bibliographystyle{plainnat}

\end{document}